\documentclass{amsart}
\usepackage{amssymb, amsmath, amsthm, graphics, comment, xspace, enumerate}
\newtheorem{thm}{Theorem}[section]

\newtheorem{prop}[thm]{Proposition}
\theoremstyle{definition}
\newtheorem{defn}[thm]{Definition}
\newtheorem{example}[thm]{Example}
\theoremstyle{remark}
\newtheorem{rem}[thm]{Remark}
\numberwithin{equation}{section}

\begin{document}
\title[Quasi-asymptotically almost periodic  functions and applications]{Quasi-asymptotically almost periodic functions and applications}

\author{Marko Kosti\' c}
\address{Faculty of Technical Sciences,
University of Novi Sad,
Trg D. Obradovi\' ca 6, 21125 Novi Sad, Serbia}
\email{marco.s@verat.net}

{\renewcommand{\thefootnote}{} \footnote{
\\ \text{  }  \ \    {\it Key words and phrases.} Quasi-asymptotically almost periodic functions, Stepanov quasi-asymptotically almost periodic functions, convolution products, evolution systems, abstract nonautonomous differential equations of first order.
\\  \text{  }  \ \ The author is partially supported by grant 174024 of Ministry
of Science and Technological Development, Republic of Serbia.}}

\begin{abstract}
The main aim of this paper is to consider the classes of quasi-asymptotically almost periodic functions and Stepanov quasi-asymptotically almost periodic functions in Banach spaces. These classes extend the well known classes of asymptotically almost periodic functions,  Stepanov asymptotically almost periodic functions and S-asymptotically $\omega$-periodic functions with values in Banach spaces. We investigate the invariance of introduced properties under the action of finite and inifinite convolution products,  
providing also an illustrative application to abstract nonautonomous differential equations of first order.
\end{abstract}

\maketitle

\section{Introduction and preliminaries}\label{intro}

Let $1\leq p<\infty,$ let $X$ and $Y$ be two non-trivial complex Banach spaces, and let
$L(X,Y)$ denote the space consisting of all continuous linear mappings from $X$ into
$Y;$ $L(X)\equiv L(X,X).$ As explained in \cite{irkutsk},
the notion of quasi-asymptotical almost periodicity is very important in the study of qualitative properties of the infinite convolution product
\begin{align}\label{wer}
{\bf F}(t):=\int^{t}_{-\infty}R(t-s)f(s)\, ds,\quad t\in {\mathbb R},
\end{align}
where $f : {\mathbb R} \rightarrow X$ is a Weyl-$p$-almost periodic function satisfying certain extra conditions and $(R(t))_{t> 0} \subseteq L(X,Y)$ is a strongly continuous operator family having a certain growth order at zero and infinity.  

It is well known that the concept introduced by H. Weyl \cite{weyl} suggests a very general way of approaching almost periodicity.  To the best knowledge of the author, the question whether the class of asymptotically Stepanov $p$-almost periodic functions, introduced by H. R. Henr\' iquez \cite{hernan1}, is contained in the class of Weyl-$p$-almost periodic functions taken without any ergodic components, has not been examined elsewhere by now. 
In this paper, we introduce the class of Stepanov $p$-quasi-asymptotically almost periodic functions and prove later that this class contains all asymptotically Stepanov $p$-almost periodic functions and make a subclass of the class consisting of all Weyl $p$-almost periodic functions (taken in the sense of  A. S. Kovanko's approach \cite{kovanko}, which is also followed in the definition of a quasi-asymptotically almost periodic function
\cite{irkutsk}).
In such a way, we initiate the study of generalized (asymptotical) almost periodicity that intermediate Stepanov and Weyl concept.

The organization and main ideas of paper are given as follows. In Subsection \ref{subs} and Subsection \ref{ramisli}, we recall the basic definitions and results about asymptotically almost periodic type functions, asymptotically almost automorphic type functions and evolution systems, Green's functions, 
respectively (it seems that a great deal of definitions introduced in Subsetion \ref{subs} is new for functions defined on the whole real axis). In Definition \ref{prc-vag}, we recall the notion of a quasi-asymptotically almost periodic (q-aap., for short) function, defined on the interval $I,$ where $I={\mathbb R}$ or $I=[0,\infty).$ After providing some observations and illustrative examples, in Theorem \ref{firstreuslt} we prove that 
any asymptotically almost automorphic (aaa.) function which is also q-aap. needs to be asymptotically almost periodic (aap.). We present a simple example of a q-aap. function that is uniformly continuous and whose range is not relatively compact in $X$ (this is a simple modification of \cite[Example 3.1]{pierro});
we also show that there exists a q-aap. function that is uniformly continuous and not aap.. The notion of a 
Stepanov $p$-q-aap. function is introduced in Definition \ref{gorilaz}, while an analogue of Theorem \ref{secundo} for Stepanov class has been proved in Theorem \ref{secundo}. The (Stepanov) class of S-asymptotically $\omega$-periodic functions, introduced by
H. R. Henr\'iquez, M. Pierri and P. T\' aboas in \cite{pierro},
is a subclass of the class consisting of the (Stepanov) class of q-aap. functions (see Proposition \ref{tabosi}).
The main structural properties of (Stepanov) q-aap. functions are proved in Theorem \ref{krew} and Proposition \ref{raw}. 
In Example 
\ref{mangup} and Example \ref{prcko-qaz}, we verify that the class of (Stepanov) q-aap. functions is not closed under pointwise products with bounded scalar-valued (Stepanov) q-aap. functions, while in Example \ref{primerusa} we show that (Stepanov) q-aap. functions do not form vector spaces equipped with the usual operations of addition and multipliplication with scalars, unfortunately.
Subsection \ref{ne-to} is devoted to the analysis of (Stepanov)
q-aap. functions depending on two parameters and related composition principles. In Theorem \ref{prvi-comp} and Theorem \ref{drugi-comp}, we analyze the composition principles for q-aap. functions depending on two parameters following the approach presented in the monograph of T. Diagana \cite{diagana} for aap. functions. The main objective in Theorem \ref{vcb-show} and Theorem \ref{vcb-primex} is to prove corresponding results for Stepanov q-aap. functions; at these places, we follow the approach of W. Long and S.-H. Ding from \cite{comp-adv}.

Concerning applications, our main results are given in Section  \ref{nedaju}, where we analyze 
the invariance of quasi-asymptotical almost periodicity under the action of convolution products, and Section \ref{convol}, where we analyze the existence and uniqueness of q-aap. solutions of 
abstract (semilinear) nonautonomous differential equations of first order. The class of q-aap. functions is maybe unique in the existing literature with regards to its invariance under the action of infinite convolution product \eqref{wer}, for the functions defined on ${\mathbb R},$ and its invariance under the action of finite convolution product 
\begin{align}\label{espebouvi}
F(t):=\int^{t}_{0}R(t-s)f(s)\, ds,\ t\geq 0,
\end{align}
for the functions defined on $[0,\infty);$ all that we need is the uniform integrability of solution operator family $(R(t))_{t> 0},$ i.e., the condition
$\int^{\infty}_{0}\|R(s)\|_{L(X,Y)}\, ds<\infty$ (Proposition \ref{finajt}, Proposition \ref{finajtt}). Similar statements hold for 
Stepanov classes of q-aap. functions, where we use a slightly stronger condition 
$\sum^{\infty}_{k=0}\|R(\cdot)\|_{L^{q}[k,k+1]}<\infty$ ($1/p+1/q=1;$ see Proposition \ref{finajt-rosi} and Proposition \ref{finajtt-rosi}). 
It is clear that the results of this section are susceptible to applications to a wide class of inhomogenous abstract
Volterra integro-differential equations and inclusions;
basically, application is possible at any place where the variation of parameters formula or some of its generalizations
takes a role. In Section \ref{convol}, we analyze the abstract nonautonomous differential equations
\begin{align}\label{nije-da-nije}
u^{\prime}(t)=A(t)u(t)+f(t),\quad t\in {\mathbb R},
\end{align}
\begin{align}\label{srq-finite}
u^{\prime}(t) = A(t)u(t) + f(t),\quad t > 0; \ u(0) = x
\end{align}
and their semilinear analogues;
here,
the operator family $A(\cdot)$ is consisted of closed linear operators with domain and range contained in $X,$
the condition (H1) clarified below holds and the evolution system
$U(\cdot, \cdot)$ generated by $A(\cdot)$ is hyperbolic, i.e, the condition (H2) clarified below holds. In Theorem \ref{jos-fajnat} (Theorem \ref{jos}), the inhomogenity $f(\cdot)$ is Stepanov $p$-q-aap. and the associated Green's function satisfies the condition \eqref{prc-nire} (\eqref{prc}). In contrast to this, in our investigations of semilinear analogues of the abstract Cauchy problems \eqref{nije-da-nije} and \eqref{srq-finite} carried out in Subsection 
\ref{polulinearni}, we assume that the corresponding function $F(\cdot,\cdot)$ is q-aap.. This is essentially caused by the fact that the Stepanov $q$-q-aap. of function $F(\cdot,x(\cdot)),$ established in composition principles Theorem \ref{vcb-show} and Theorem \ref{vcb-primex}, holds only if we additionally assume that the range of function $x(\cdot)$ is relatively compact, which need not be true for q-aap. functions and their Stepanov generalizations. Further study of semilinear nonautonomous differential equations with forcing term $F(\cdot,\cdot)$ belonging to Stepanov class of q-aap. functions is without scope of this paper. Finally, in Example \ref{dirichlet-jazz}, we provide an instructive example of applications of our abstract theoretical results obtained,
continuing thus the analyses raised by T. Diagana in \cite[Section 4]{diagana-2008} and the author \cite[Example 3.1]{aot-besic}.

We use the standard notation throughout the paper. By $\|x\|,$ $\|y\|_{Y}$ and $\|T\|_{L(X,Y)}$ we denote the norm of an element $x\in X,$ an element $y\in Y$ and a continuous linear mapping $T\in L(X,Y).$ Let $I={\mathbb R}$ or $I=[0,\infty)$ in the sequel.
By $C_{b}(I : X),$ $C_{0}(I : X)$ and $BUC(I:X)$ we denote the vector spaces consisting of all bounded continuous functions $f  : I \rightarrow X,$ all bounded continuous functions $f  : I \rightarrow X$ such that $\lim_{|t|\rightarrow +\infty}\|f(t)\|=0$ and all bounded uniformly continuous functions $f  : I \rightarrow X,$ respectively. If $A$ is a closed linear operator
acting on $X,$
then the domain, kernel space, range and resolvent set of $A$ will be denoted by
$D(A),$ $N(A),$ $R(A)$ and $\rho(A),$
respectively; $R(\lambda :A) \equiv (\lambda-A)^{-1}$ ($\lambda \in \rho(A)$). Since no confusion
seems likely, we will identify $A$ with its graph. For any number $t\in {\mathbb R},$ we define $\lfloor t \rfloor :=\sup \{
l\in {\mathbb Z} : t \geq l \}$ and $\lceil t \rceil :=\inf \{ l\in
{\mathbb Z} : t\leq l \}.$
By $\chi_{B}(\cdot)$ we denote the characteristic function of set $B.$

\subsection{Asymptotically almost periodic type functions, asymptotically almost automorphic type functions and their generalizations}\label{subs}

Let $f : I \rightarrow X$ be continuous. Given $\epsilon>0,$ we call $\tau>0$ an $\epsilon$-period for $f(\cdot)$ iff
$
\| f(t+\tau)-f(t) \| \leq \epsilon, $ $ t\in I.
$
By $\vartheta(f,\epsilon)$ we denote the set consisted of all $\epsilon$-periods for $f(\cdot).$ It is said that $f(\cdot)$ is almost periodic (ap.) iff for each $\epsilon>0$ the set $\vartheta(f,\epsilon)$ is relatively dense in $I,$ which means that
there exists $l>0$ such that any subinterval of $I$ of length $l$ meets $\vartheta(f,\epsilon)$. The space consisted of all almost periodic functions from the interval $I$ into $X$ will be denoted by $AP(I:X).$

The class of asymptotically almost periodic functions was introduced by M. Fr\' echet in 1941, for the case that $I=[0,\infty)$ (more details about aap. functions with values in Banach spaces can be found in \cite{cheban}-\cite{diagana}, \cite{gaston} and references cited therein). If $I={\mathbb R},$ there is several non-equivalent notions of
an aap. function. Here we follow the approach of C. Zhang \cite{zhang}:

A function $f \in C_{b}(I : X)$ is said to asymptotically almost periodic iff
for every $\epsilon >0$ we can find numbers $ l > 0$ and $M >0$ such that every subinterval of $I$ of
length $l$ contains, at least, one number $\tau$ such that $\|f(t+\tau)-f(t)\| \leq \epsilon$ provided $|t|,\ |t+\tau| \geq M.$
The space consisting of all aap. functions from $I$ into $X$ is denoted by
$AAP(I : X).$ For a function $f \in C_{b}(I:X),$ the following
statements are equivalent (see \cite{RUESS} for the case that $I=[0,\infty)$ and \cite[Theorem 2.6]{zhang} for the case that $I={\mathbb R}$):
\begin{itemize}
\item[(i)] $f\in AAP(I :X).$
\item[(ii)] There exist uniquely determined functions $g \in AP(I :X)$ and $\phi \in  C_{0}(I: X)$
such that $f = g+\phi.$
\end{itemize}

Unless stated otherwise, in the sequel we will always assume that $1\leq p <\infty.$ Let $l>0,$ and let $f,\ g\in L^{p}_{loc}(I :X).$ We define the Stepanov `metric' by
\begin{align*}
D_{S_{l}}^{p}\bigl[f(\cdot),g(\cdot)\bigr]:= \sup_{x\in I}\Biggl[ \frac{1}{l}\int_{x}^{x+l}\bigl \| f(t) -g(t)\bigr\|^{p}\, dt\Biggr]^{1/p}.
\end{align*}
The Stepanov `norm' of $f(\cdot)$ is defined by
$
\| f  \|_{S_{l}^{p}}:= D_{S_{l}}^{p}[f(\cdot),0].
$
It is said that a function $f\in L^{p}_{loc}(I :X)$ is Stepanov $p$-bounded, $S^{p}$-bounded shortly, iff
$
\|f\|_{S^{p}}:=\sup_{t\in I}( \int^{t+1}_{t}\|f(s)\|^{p}\, ds)^{1/p}<\infty.
$
The space $L_{S}^{p}(I:X)$ consisted of all $S^{p}$-bounded functions becomes a Banach space equipped with the above norm.
We say that a function $f\in L_{S}^{p}(I:X)$ is Stepanov $p$-almost periodic, $S^{p}$-ap. shortly, iff the function
$
\hat{f} : I \rightarrow L^{p}([0,1] :X),
$ defined by
$
\hat{f}(t)(s):=f(t+s),\quad t\in I,\ s\in [0,1]
$
is ap..
It is said that $f\in  L_{S}^{p}(I: X)$ is asymptotically Stepanov $p$-almost periodic, $S^{p}$-aap. shortly, iff $\hat{f} : I \rightarrow L^{p}([0,1]:X)$ is aap..
By $APS^{p} (I: X)$ and $AAPS^{p} (I: X)$ we denote the spaces consisted of all $S^{p}$-ap. functions $I\mapsto X$ and $S^{p}$-aap. functions $I\mapsto X,$ respectively.

The notion of an (equi-)Weyl almost periodic function is given as follows (see \cite{deda}, \cite{nova-mono} and references cited therein for more details on the subject):

\begin{defn}\label{weyl-defn}
Let $1\leq p<\infty$ and $f\in L_{loc}^{p}(I: X).$ \index{function!equi-Weyl-$p$-almost periodic} \index{function!Weyl-$p$-almost periodic}
\begin{itemize}
\item[(i)] We say that the function $f(\cdot)$ is equi-Weyl-$p$-almost periodic, $f\in e-W_{ap}^{p}(I:X)$ for short, iff for each $\epsilon>0$ we can find two real numbers $l>0$ and $L>0$ such that any interval $I'\subseteq I$ of length $L$ contains a point $\tau \in  I'$ such that
\begin{align*}
\sup_{x\in I}\Biggl[ \frac{1}{l}\int_{x}^{x+l}\bigl \| f(t+\tau) -f(t)\bigr\|^{p}\, dt\Biggr]^{1/p} \leq \epsilon, \mbox{ i.e., } D_{S_{l}}^{p}\bigl[f(\cdot+\tau),f(\cdot)\bigr] \leq \epsilon.
\end{align*}
\item[(ii)] We say that the function $f(\cdot)$ is Weyl-$p$-almost periodic, $f\in W_{ap}^{p}(I: X)$ for short, iff for each $\epsilon>0$ we can find a real number $L>0$ such that any interval $I'\subseteq I$ of length $L$ contains a point $\tau \in  I'$ such that
\begin{align*}
\lim_{l\rightarrow \infty} \sup_{x\in I}\Biggl[ \frac{1}{l}\int_{x}^{x+l}\bigl \| f(t+\tau) -f(t)\bigr\|^{p}\, dt\Biggr]^{1/p} \leq \epsilon, \mbox{ i.e., } \lim_{l\rightarrow \infty}D_{S_{l}}^{p}\bigl[f(\cdot+\tau),f(\cdot)\bigr] \leq \epsilon.
\end{align*}
\end{itemize}
\end{defn}

We also need the definition of an asymptotically almost automorphic function defined on the interval $I.$ For beginning, let us recall that
a continuous function $f : {\mathbb R} \rightarrow X$ is said to be
almost automorphic (aa., for short) iff for every real sequence $(b_{n})$ there exist a subsequence $(a_{n})$ of $(b_{n})$ and a map $g : {\mathbb R} \rightarrow X$ such that
$
\lim_{n\rightarrow \infty}f( t+a_{n})=g(t)\ \mbox{ and } \  \lim_{n\rightarrow \infty}g\bigl( t-a_{n}\bigr)=f(t),
$
pointwise for $t\in {\mathbb R}.$ Any aa. function $f(\cdot)$ needs to be bounded and the following supremum formula holds (see e.g. \cite[Lemma 3.9.9]{nova-mono}):
$$
\|f\|_{\infty}:=\sup_{x\in {\mathbb R}}\|f(x)\|=\sup_{x\geq a}\|f(x)\|\ \ \mbox { for any number }\ \ a\in {\mathbb R}.
$$

In this paper, we will use the following notion (see also \cite[Definition 2.3]{din-man}):

\begin{defn}
\begin{itemize}
\item[(i)]
A bounded continuous function $f : {\mathbb R} \rightarrow X$ is said to be asymptotically almost automorphic iff there exist two functions $h\in AA({\mathbb R} : X)$ and $q\in C_{0}({\mathbb R} : X)$ such that $f=h+q$ on ${\mathbb R}$.
\item[(ii)] A bounded continuous function $f : [0,\infty) \rightarrow X$ is said to be asymptotically almost automorphic iff there exist two functions $h\in AA({\mathbb R} : X)$ and $q\in C_{0}([0,\infty) : X)$ such that $f=h+q$ on $[0,\infty)$.
\end{itemize}
\end{defn}

It is well known that any (asymptotically) almost periodic function defined on the interval $I$ is (asymptotically) almost automorphic as well as that the converse statement does not hold in general.
Following G. M. N'Gu\' er\' ekata and A. Pankov \cite{gaston-apankov}, a function $f\in L_{loc}^{p}({\mathbb R}:X)$ is called Stepanov $p$-almost automorphic, $S^{p}$-aa. for short, iff for
every real sequence $(a_{n}),$ there exists a subsequence $(a_{n_{k}})$ \index{ Stepanov almost automorphy}
and a function $g\in L_{loc}^{p}({\mathbb R}:X)$ such that
\begin{align*}
\lim_{k\rightarrow \infty}\int^{t+1}_{t}\Bigl \| f\bigl(a_{n_{k}}+s\bigr) -g(s)\Bigr \|^{p} \, ds =0
\mbox{ 
and }
\lim_{k\rightarrow \infty}\int^{t+1}_{t}\Bigl \| g\bigl( s-a_{n_{k}}\bigr) -f(s)\Bigr \|^{p} \, ds =0
\end{align*}
for each $ t\in {\mathbb R}.$ By $ AAS^{p}({\mathbb R} : X)$ we denote the vector space consisting of all $S^{p}$-aa. functions ${\mathbb R} \rightarrow X$.

The following definition seems to be new in case $I={\mathbb R}:$
\begin{defn}
\begin{itemize}
\item[(i)]
An $S^{p}$-bounded function $f : {\mathbb R} \rightarrow X$ is said to be asymptotically Stepanov $p$-almost automorphic, $S^{p}$-aaa. for short, iff there exist two functions $h\in AAS^{p}({\mathbb R} : X)$ and an $S^{p}$-bounded function $q : {\mathbb R} \rightarrow X$
such that
$\hat{q}\in C_{0}({\mathbb R} : L^{p}([0,1]: X))$ and $f=h+q$ a.e. on ${\mathbb R}$.
\item[(ii)] An $S^{p}$-bounded function $f : [0,\infty) \rightarrow X$ is said to be asymptotically Stepanov $p$-almost automorphic iff there exist two functions $h\in AAS^{p}({\mathbb R} : X)$ and an $S^{p}$-bounded function $q : [0,\infty) \rightarrow X$
such that
$\hat{q}\in C_{0}([0,\infty) : L^{p}([0,1]: X))$ and $f=h+q$ a.e. on $[0,\infty)$.
\end{itemize}
By $ AAAS^{p}(I : X)$ we denote the vector space consisting of all asymptotically $S^{p}$-almost automorphic functions $I \rightarrow X$.
\end{defn}

It can be easily verified that the (asymptotical) $S^{p}$-almost automorphy of $f(\cdot)$ implies the (asymptotical) almost automorphy of the mapping
$
\hat{f} : I \rightarrow L^{p}([0,1] :X)
$ defined above. Any
(asymptotically) $S^{p}$-almost periodic function $f : I \mapsto X$ has to be (asymptotically) $S^{p}$-almost automorphic, while the converse statement does not hold in general.

Denote by $C_{\omega}(I: X)$ the space consisting of all continuous $\omega$-periodic functions $g: I\rightarrow X.$
The following two definitions as well Definition \ref{dimb} seem to be new in case $I={\mathbb R},$ likewise (see H. R. Henr\'iquez et al. \cite{pierro},  H. R. Henr\'iquez \cite{hrh}, W. Dimbour, S. M. Manou-Abi \cite{dimbour} for case $I=[0,\infty)$):

\begin{defn}\label{henriqz}
Let $\omega \in I.$ Then we say that a bounded continuous function $f : I \rightarrow X$ is S-asymptotically $\omega$-periodic iff $\lim_{|t|\rightarrow \infty}\|f(t+\omega)-f(t)\|=0.$ Denote by $SAP_{\omega}(I: X)$ the space consisting of all such functions. 
\end{defn}

\begin{defn}\label{cime}
Let $\omega \in I.$ 
A bounded continuous function $f : I \rightarrow X$ is said to be asymptotically $\omega$-almost periodic iff there exists a function $g\in C_{\omega}(I: X)$ and a function $q\in C_{0}(I: X)$ such that $f(t)=g(t)+q(t)$ for all $t\in I.$ Denote by $AP_{\omega}(I: X)$ the vector space consisting of all such functions.
\end{defn}

It is straightforward to see that $AP_{\omega}(I: X)\subseteq SAP_{\omega}(I: X)$ and the inclusion is strict.
We will also work with the class of Stepanov $S^{p}$-asymptotically $\omega$-periodic functions. 

\begin{defn}\label{dimb} 
Let $\omega \in I.$ A Stepanov $p$-bounded function $f(\cdot)$ is said to be Stepanov $p$-asymptotically $\omega$-periodic iff 
$$
\lim_{|t|\rightarrow \infty}\int^{t+1}_{t}\bigl\|f(s+\omega)-f(s)\bigr\|^{p}\, ds=0.
$$ 
Denote by $S^{p}SAP_{\omega}(I: X)$ the space consisting of all such functions.
\end{defn}

We have that $SAP_{\omega}(I: X)\subseteq S^{p}SAP_{\omega}(I: X)$ and the inclusion is strict.

\subsection{Evolution systems and Green's functions}\label{ramisli}

The following definition is well known in the existing literature:

\begin{defn}\label{paz}
A family $\{U(t, s) : t \geq s,\  t,\  s  \in {\mathbb R}\} \subseteq L(X)$ is said to be
an evolution system iff the following holds:
\begin{itemize}
\item[(a)] $U(s, s) = I,$ $U(t, s) = U(t, r)U(r, s)$ for $t \geq  r \geq  s$ and $t,\ r,\ s  \in {\mathbb R},$
\item[(b)] $\{(\tau, s) \in {\mathbb R}^{2} : \tau > s \} \ni (t, s) \mapsto U(t, s)x$ is continuous for any fixed element $x\in X$.
\end{itemize}
\end{defn}

In the sequel, it will be always assumed that the family $A(\cdot)$
satisfies the following condition introduced by  P. Acquistapace and B. Terreni in \cite{a-t} (with $\omega=0$):
\begin{itemize}
\item[(H1):] There is a number $\omega \geq 0$ such that the family of closed linear operators $A(t),$ $t\in {\mathbb R}$ on $X$ satisfies $\overline{\Sigma_{\phi}} \subseteq \rho(A(t)-\omega),$
$$
\bigl\| R(\lambda : A(t)-\omega) \bigr\| =O\Bigl( \bigl(1+|\lambda|\bigr)^{-1} \Bigr),\quad t\in {\mathbb R},\ \lambda \in \overline{\Sigma_{\phi}},\mbox{ and }
$$
$$
\Bigl\| (A(t)-\omega) R(\lambda : A(t)-\omega) \bigl[ R(\omega : A(t))- R(\omega : A(s)) \bigr] \Bigr\| =O\Bigl( |t-s|^{\mu}|\lambda|^{-\nu} \Bigr),
$$
for any $t,\ s\in {\mathbb R},\ \lambda \in \overline{\Sigma_{\phi}},$
where $\phi \in (\pi/2,\pi),$ $0<\mu,\ \nu \leq 1$ and $\mu+\nu>1.$
\end{itemize}

Then we know that there exists an evolution system
$U(\cdot, \cdot)$ generated by $A(\cdot),$ satisfying that $\|U(t,s)\|=O(1)$ for $t\geq s,$ as well as a great deal of other conditions (\cite{a-t}).
Besides (H1), we will also assume that the following condition holds:
\begin{itemize}
\item[(H2):]
The evolution system $U(\cdot, \cdot)$ generated by $A(\cdot)$ is hyperbolic (or, equivalently, has exponential dichotomy), i.e.,
there exist a family of projections $(P(t))_{t \in {\mathbb R}}\subseteq L(X),$ being uniformly bounded and strongly continuous
in $t,$ and constants $M',\ \omega > 0$ such that (a)-(c) holds with $Q := I - P$ and $Q(\cdot):=I-P(\cdot),$ where $I$ stands for the identity operator on $X$ and:
\begin{itemize}
\item[(a)] $U(t, s)P(s) = P(t)U(t, s)$ for all $t \geq s,$
\item[(b)] the restriction $U_{Q}(t, s) : Q(s)X \rightarrow Q(t)X$ is invertible for all $t \geq s$ (here we define
$U_{Q}(s, t) = U_{Q}(t, s)^{-1}$),
\item[(c)] $\|U(t, s)P(s)\| \leq  M'e^{-\omega (t-s)}$ and $\|U_{Q}(s, t)Q(t)\|\leq M'e^{-\omega (t-s)}$ for all $t \geq s.$
\end{itemize}
\end{itemize}

It is said that $U(\cdot,\cdot)$ is
exponentially stable iff the choice $P(t) = I$ for all $t \in {\mathbb R}$ can be made; $U(\cdot,\cdot)$ is said to be (bounded) exponentially bounded iff there exist
two finite real constants $M > 0$ and ($\omega=0$) $\omega \in {\mathbb R}$
such that $\|U(t, s)P(s)\| \leq  Me^{-\omega (t-s)}$ for all $t \geq s.$
The associated Green's function $\Gamma(\cdot,\cdot)$ is defined through
\[
\Gamma(t,s):=\left\{
\begin{array}{l}
U(t,s)P(s),\  t\geq s,\ t,\ s\in {\mathbb R},\\
-U_{Q}(t,s)Q(s),\ t< s,\ t,\ s\in {\mathbb R}.
\end{array}
\right.
\]
Let $M'$ be the constant appearing in (H2). Then
\begin{align}\label{srq}
\|\Gamma(t,s)\| \leq M'e^{-\omega |t-s| },\quad t,\ s\in {\mathbb R}
\end{align}
and the function
\begin{align}\label{do-koske}
u(t):=\int^{+\infty}_{-\infty}\Gamma(t,s)f(s)\, ds,\quad t\in {\mathbb R}
\end{align}
is a unique bounded continuous function on ${\mathbb R}$ satisfying
$$
u(t)=U(t, s)u(s) +
\int^{t}_{s}U(t, \tau )f(\tau )\, d\tau, \quad  t \geq s;
$$
cf. \cite{schnaubelt}. In the sequel, it will be said that $u(\cdot)$ is a
mild solution of the abstract Cauchy problem \eqref{nije-da-nije}.

Let $f : [0,\infty)\rightarrow X$ be continuous. By a mild solution of the abstract Cauchy problem \eqref{srq-finite}
we mean the function
\begin{align}\label{ndn-finite}
u(t):=U(t,0)x+\int^{t}_{0}U(t,s)f(s)\, ds,\quad t\geq 0.
\end{align}
For more details on the subject, we refer the reader to \cite[Section 5]{schnaubelt}.

We will also consider the following semilinear Cauchy problems:
\begin{align}\label{petar-ghj}
u^{\prime}(t)=A(t)u(t)+F(t,u(t)),\quad t\in {\mathbb R}
\end{align}
and
\begin{align}\label{petar-ghj-finite}
u^{\prime}(t)=A(t)u(t)+F(t,u(t)),\quad t>0; \ u(0) = x.
\end{align}

Let the space $S^{p}Q-AAP(I \times X : X)$ be defined as in Subsection \ref{ne-to}, and let $F \in S^{p}Q-AAP(I \times X : X)  .$ 

\begin{defn}\label{leton}
\begin{itemize}
\item[(i)] A function $u \in C_{b}({\mathbb R}:X)$ is said to be a mild solution of \eqref{petar-ghj} iff
\begin{align*}
u(t)=\int^{+\infty}_{-\infty}\Gamma(t,s)F(s,u(s))\, ds,\quad t\in {\mathbb R}.
\end{align*}
\item[(ii)] A function $u \in C_{b}([0,\infty ):X)$ is said to be a mild solution of \eqref{petar-ghj-finite} iff
\begin{align*}
u(t)=U(t,0)x+\int^{t}_{0}U(t,s)F(s,u(s))\, ds,\quad t\geq 0.
\end{align*}
\end{itemize}
\end{defn}

\section{Quasi-asymptotically almost periodic functions and their generalizations}\label{profica}

We start by recalling the following definition (\cite{irkutsk}):

\begin{defn}\label{prc-vag}
Suppose that $I=[0,\infty)$ or $I={\mathbb R}.$ Then we say that a bounded continuous function $f : I\rightarrow X$ is quasi-asymptotically almost periodic iff 
for each $\epsilon>0$ there exists a finite number $L(\epsilon)>0$ such that any interval $I'\subseteq I$ of length $L(\epsilon)$ contains at least one number $\tau \in I'$ satisfying that
there exists a finite number $M(\epsilon,\tau)>0$ such that
\begin{align}\label{dub garden}
\| f(t+\tau)-f(t)\|\leq \epsilon,\mbox{ provided }t\in I\mbox{ and } |t|\geq M(\epsilon,\tau).
\end{align}
Denote by $Q-AAP(I : X)$ the set consisting of all quasi-asymptotically almost periodic functions from $I$ into $X.$
\end{defn}

In order to avoid unnecessary repeating, we will use the shorthand
\begin{itemize}
\item[(S):]``there exists a finite number $L(\epsilon)>0$ such that any interval $I'\subseteq I$ of length $L(\epsilon)$ contains at least one number $\tau \in I'$ satisfying that
there exists a finite number''.
\end{itemize}

\begin{rem}\label{dubgarden}
It is not relevant whether we will write \eqref{dub garden} or
\begin{align*}
\| f(t+\tau)-f(t)\|\leq \epsilon,\mbox{ provided }t\in I,\ |t|\geq M(\epsilon,\tau) \mbox{ and } |t+\tau|\geq M(\epsilon,\tau).
\end{align*}
Using this observation, it can be easily seen that
the
class of aap. functions is contained in the class of q-aap. functions (the number $M$ depends only on $\epsilon$ and not on $\tau$ for aap. functions).  The converse statement is not true, however:
Let $I={\mathbb R}$ and let $f(\cdot)$ be any bounded scalar-valued continuous function such that $f(t)=1$ for all $t\geq 0$ and $f(t)=0$ for all $t\leq -1.$ Then $f (\cdot)$ is q-aap., not aap. and not equi-Weyl-$p$-ap. for any $p\in [1,\infty)$ (\cite{deda}, \cite{irkutsk}). Applying Theorem \ref{firstreuslt}(i) below we can see that $f(\cdot)$ is not aaa., as well.
\end{rem}

We  continue by providing an illustrative example and one more remark.

\begin{example}\label{dubgarden121}
Suppose that $f\in C^{1}(I: X) \cap C_{b}(I:X)$ and $f^{\prime}\in C_{0}(I:X).$ Then $f\in  Q-AAP(I : X).$ In order to see this, it suffices to apply the Langrange mean value theorem as well as to choose, in Definition \ref{prc-vag}, $L(\epsilon)>0$ arbitrarily and any $\tau \neq 0$ from an arbitrary interval $I'$ of length $L(\epsilon).$ Then, for this $\epsilon>0$ and $\tau \in I',$ there exists a sufficiently large $M(\epsilon,\tau)>0$ such that $[t,t+\tau]\subseteq \{s \in I: |s|\geq M_{0}(\epsilon,\tau)\}$ for $|t|\geq M(\epsilon,\tau),$ where $M_{0}(\epsilon,\tau)$ is already chosen so that $\|f^{\prime}(\xi)\|\leq \epsilon/|\tau|$ for $|\xi|\geq M_{0}(\epsilon,\tau);$ then we have
$$
\|f(t+\tau)-f(t)\|\leq |\tau| \sup_{\xi \in [t,t+\tau]}\|f^{\prime}(\xi)\|\leq \epsilon,\ |t|\geq M(\epsilon,\tau).
$$
It is worth noting that there exists a function $f(\cdot)$ that is not aap. and satisfies the above properties; a typical example is given by $f(t):=\sin (\ln (1+t)),$ $t\geq 0$ (see also \cite[Example 4.1, Theorem 4.2]{rozpr}).
\end{example}

\begin{rem}\label{anti-periodic}
In our joint research paper \cite{krag} with D. Velinov, we have recently introduced and analyzed the class of (asymptotically) almost anti-periodic functions. The notion of
quasi-asymptotically almost anti-periodicity (q-aanp., for short) can be introduced in the following way: A bounded continuous
$f : I\rightarrow X$ is called quasi-asymptotically almost anti-periodic iff for each $\epsilon>0$ (S) holds with a number $M(\epsilon,\tau)>0$ such that
\begin{align}\label{r-basara}
\| f(t+\tau)+f(t)\|\leq \epsilon,\mbox{ provided }t\in I\mbox{ and } |t|\geq M(\epsilon,\tau).
\end{align}
Suppose that $t,\ t+\tau \in I$ as well as that $|t|\geq M(\epsilon, \tau)+|\tau|.$ Then $|t+\tau|\geq M(\epsilon, \tau)$ and applying \eqref{r-basara} twice, we get that
\begin{align*}
\| f(t+2\tau) & -f(t)\| =\bigl\| [f(t+2\tau)+f(t+\tau)]-[f(t+\tau)+f(t)]\bigr\|
\\ & \leq \| f(t+2\tau) +f(t+\tau)\| +\|f(t+\tau) +f(t)\|\leq 2 \epsilon .
\end{align*}
Hence, any q-aanp. function is automatically q-aap.. Further analysis of q-aanp. functions and their Stepanov generalizations are without scope of this paper.
\end{rem}

The space $C_{b}(I: X) \setminus Q-AAP(I : X)$
is sufficiently large; it is clearly non-empty because
it is very plainly to construct an example of an infinite-differentiable bounded function $f : I \rightarrow {\mathbb C}$ such that for each number $\tau \in I$ there exists a sequence $(t_{n})_{n\in {\mathbb N}}$ in $I$ with the properties that $\lim_{n\rightarrow \infty}|t_{n}|=\infty$ and $|f(t_{n}+\tau)-f(t_{n})|\geq 1$ for all $n\in {\mathbb N}.$ Furthermore, we have the following:

\begin{thm}\label{firstreuslt}
\begin{itemize}
\item[(i)] $AAA(I : X) \cap Q-AAP(I : X)=AAP(I: X)$ and $[AAA(I : X) \setminus AAP(I : X)] \cap Q-AAP(I:X)=\emptyset.$
\item[(ii)] $AA({\mathbb R}: X) \cap Q-AAP({\mathbb R}: X)=AP({\mathbb R}: X).$
\end{itemize}
\end{thm}

\begin{proof}
For the sake of brevity, we will consider only the case that $I={\mathbb R}.$ It is clear that
$AAP({\mathbb R}: X) \subseteq AAA({\mathbb R}: X) \cap Q-AAP({\mathbb R}: X).$ To prove the converse inclusion,
suppose that $f\in AAA({\mathbb R}: X) \cap Q-AAP({\mathbb R}: X).$ Then there exist two functions $h\in AA({\mathbb R} : X)$ and $q\in C_{0}({\mathbb R} :  X)$ such that $f=h+q$ on ${\mathbb R}$ and for each $\epsilon>0$ (S) holds with a number $M(\epsilon,\tau)>0$ such that
\begin{align}\label{dub gardenpos}
\| [h(t+\tau)-h(t)]+[q(t+\tau)-q(t)]\|\leq \epsilon,\mbox{ provided }t\in {\mathbb R}\mbox{ and } |t|\geq M(\epsilon,\tau).
\end{align}
Fix a number $\epsilon>0$ and suppose that the real number $\tau$ satisfies \eqref{dub gardenpos} for $|t|\geq M(\epsilon,\tau).$ Since $q\in C_{0}({\mathbb R} :  X),$ we have that there exists a finite number
$M_{1}(\epsilon, \tau)\geq  M(\epsilon,\tau)$ such that
\begin{align}\label{dub gardenposh}
\| h(t+\tau)-h(t)\|\leq 2\epsilon,\mbox{ provided }t\in {\mathbb R}\mbox{ and } |t|\geq M_{1}(\epsilon,\tau).
\end{align}
Define the function $H : {\mathbb R} \rightarrow X$ by $H(t):=h(t+\tau)-h(t),$ $t\in {\mathbb R}.$ Since the space $AA({\mathbb R}: X)$ is translation invariant, we have $H\in AA({\mathbb R}: X).$ Applying supremum formula and \eqref{dub gardenposh}, we get
$$
\sup_{t\in {\mathbb R}}\|H(t)\| =\sup_{t\geq M_{1}(\epsilon, \tau)}\|H(t)\|=\sup_{t\geq M_{1}(\epsilon, \tau)}\|h(t+\tau)-h(t)\|\leq 2\epsilon.
$$
Hence, $\|h(t+\tau)-h(t)\|\leq 2\epsilon$ for all $t\in {\mathbb R}$ and $h(\cdot)$ is ap. by definition. Hence, $AAP({\mathbb R}: X) = AAA({\mathbb R}: X) \cap Q-AAP({\mathbb R}: X),$ which immediately implies the second equality in (i). The proof of (ii) follows from the above arguments, as well.
\end{proof}

It is expected that the range of a function $f\in Q-AAA(I: X) \cap BUC(I : X)$ need not be relatively compact, as in the case of aap. functions. In the following example, we will explain this fact in case $I=[0,\infty):$

\begin{example}\label{gent}
Let $X:=c_{0}({\mathbb C}).$ Although the final conclusions presented here holds for the function $f(\cdot)$ considered in \cite[Example 3.1]{pierro}, we will prove that the range of function
$$
f(t):=\Biggl(\frac{4n^{2}t^{2}}{(t^{2}+n^{2})^{2}} \Biggr)_{n\in {\mathbb N}},\ t\geq 0
$$
is not relatively compact in $X.$ Using a similar analysis as in the afore-mentioned example, we get that the function $f(\cdot)$ is bounded and uniformly continuous with the estimate 
$
\|f(t+s)-f(t)\|\leq 8s,$ $t,\ s>0$
holding true. Further on, for each $t>0$ and $\tau \geq 0,$ we have:
\begin{align*}
\|f(t+\tau)-f(t)\|
 \leq  & \sup_{n\in {\mathbb N}} \frac{4n^{2}\bigl[ (t+\tau)^{2}+\tau^{2}\bigr]}{(t^{2}+n^{2})^{2}((t+\tau)^{2}+n^{2})^{2}}\leq t^{-4}+4\frac{\tau^{2}}{t^{2}},\ t>0,\ \tau\geq 0.
\end{align*}
By \cite[Remark 3.1, Proposition 3.3]{pierro}, it readily follows that the range of $f(\cdot)$ is not 
relatively compact as well as that there is no number $\tau>0$ such that $f(\cdot)$ is $\tau$-normal on compact subsets; see \cite[Definition 3.2]{pierro} for the notion.
\end{example}

Now we will prove that the existence of a number $\omega \in I$ such that $f : I \rightarrow X$ is S-asymptotically $\omega$-periodic implies that $f(\cdot)$ is qaap.:

\begin{prop}\label{tabosi}
We have $SAP_{\omega}(I: X)\subseteq Q-AAP(I:X).$
\end{prop}

\begin{proof}
For $\epsilon>0$ given in advance, we can take $L(\epsilon)=2\omega.$ Then any interval $I'\subseteq I$ of length $L(\epsilon)$ contains a number $\tau=n\omega$ for some $n\in {\mathbb N}_{0}.$ For this $n$ and $\epsilon,$ there exists a finite number $M(\epsilon,n)>0$
such that $\|f(t+\omega)-f(t)\|\leq \epsilon/n\omega$ for $|t|\geq M(\epsilon,n).$ Then the final conclusion
follows from the estimates
\begin{align*}
\|f(t+n\omega)-f(t)\|\leq \sum_{k=0}^{n-1}\|f(t+\tau-k\omega)-f(t+\omega-(k+1)\omega)\|\leq n\epsilon/n\omega=\epsilon,
\end{align*}
provided $|t|\geq M(\epsilon,n)+n\omega.$
\end{proof}

In \cite[Example 17]{xie}, R. Xie and C. Zhang
have constructed an example of a function $f\in SAP_{2}([0,\infty) : X)$ that is not uniformly continuous. By the above proposition, the function $f(\cdot)$ is q-aap. and not uniformly continuous.

The following simple proposition, already known in the case that $I=[0,\infty),$ can be deduced by using the arguments contained in the proof 
of \cite[Proposition 3.6]{prc-marek}. An alternative proof can be given by using Theorem \ref{firstreuslt}, Proposition \ref{tabosi}
and an easy reformulation of \cite[Lemma 3.1]{pierro} in case $I={\mathbb R}$:

\begin{prop}\label{nula-sam}
Let $\omega \in I.$
\begin{itemize}
\item[(i)] Suppose that $f\in SAP_{\omega}(I: X) \cap AAA(I : X).$ Then $f\in AP_{\omega}(I:X).$
\item[(ii)] Suppose that $f\in SAP_{\omega}(I: X) \cap AA(I : X).$ Then $f\in C_{\omega}(I:X).$
\end{itemize}
\end{prop}

Now we will introduce the Stepanov generalization of q-aap. functions:

\begin{defn}\label{gorilaz}
Let $f\in L_{S}^{p}(I:X).$ Then it is said $f(\cdot)$ is Stepanov $p$-quasi-asymptotically almost periodic ($S^{p}$-qaap., for short) iff for each $\epsilon>0$ (S) holds with a number $M(\epsilon,\tau)>0$ such that
\begin{align}\label{primer}
\int^{t+1}_{t}\|f(s+\tau)-f(s)\|^{p}\, ds \leq \epsilon^{p}, & \mbox{ provided }t\in I\mbox{ and } |t|\geq M(\epsilon,\tau).
\end{align}
Denote by $S^{p}Q-AAP(I : X)$ the set consisting of all Stepanov $p$-quasi-asymptotically almost periodic functions from $I$ into $X.$
\end{defn}

Using only definition, it readily follows that
$
Q-AAP(I : X)\subseteq S^{p}Q-AAP(I : X) ;
$
it is clear that this inclusion can be strict since the function $f(t):=\chi_{[-1,\infty)}(t),$ $t\in {\mathbb R}$ is in class $S^{p}Q-AAP({\mathbb R} : X)$ but not in class $Q-AAP({\mathbb R} : X)$ because $f(\cdot)$ is not continuous. Furthermore, it follows immediately from definition that any $S^{p}$-aap. function
is $S^{p}$-qaap. so that $AAPS^{p} (I: X) \subseteq S^{p}Q-AAP(I : X);$
this inclusion can be also strict since the continuous function $f(\cdot)$ from Remark \ref{dubgarden} is not $S^{p}$-aap.. Furthermore, if $1\leq p <p'<\infty,$
then
$S^{p'}Q-AAP(I : X)\subseteq S^{p}Q-AAP(I : X)$ and for any function $f\in L_{S}^{p}(I:X),$ we have that $f(\cdot)$ is $S^{p}$q-aap. iff the function $\hat{f} : I \rightarrow L^{p}([0,1] : X)$ is q-aap.. Using this fact and Theorem \ref{firstreuslt}, we can simply verify the validity of following result:

\begin{thm}\label{secundo}
\begin{itemize}
\item[(i)] $S^{p}AAA(I : X) \cap S^{p}Q-AAP(I : X)=S^{p}AAP(I: X)$ and $[S^{p}AAA(I : X) \setminus S^{p}AAP(I : X)] \cap S^{p}Q-AAP(I:X)=\emptyset.$
\item[(ii)] $S^{p}AA({\mathbb R}: X) \cap S^{p}Q-AAP({\mathbb R}: X)=S^{p}AP({\mathbb R}: X).$
\end{itemize}
\end{thm}

The proof of following result is very similar to that of Proposition \ref{tabosi} and therefore omitted:

\begin{prop}\label{profice}
We have $S^{p}SAP_{\omega}(I: X)\subseteq S^{p}Q-AAP(I:X).$
\end{prop}

After introduction of Definition \ref{gorilaz}, it seems reasonable to ask whether we can analyze Weyl and Besicovitch generalizations of q-aap. functions. The following result says that the space $ S^{p}Q-AAP(I : X)$ is contained in $W_{ap}^{p}(I: X)$ and that the above question is more or less ridicolous (we have already seen that $ S^{p}Q-AAP(I : X)\nsubseteq  e-W_{ap}^{p}(I: X)$):

\begin{prop}\label{weylova-klasa}
We have $ S^{p}Q-AAP(I : X)\subseteq W_{ap}^{p}(I: X).$
\end{prop}

\begin{proof}
Let $\epsilon>0$ be given. Then (S) holds with a number $M(\epsilon,\tau)>0$ such that \eqref{primer} is satisfied. We need to estimate the term
\begin{align}\label{vasq}
\sup_{x\in I}\Biggl[ \frac{1}{l}\int_{x}^{x+l}\bigl \| f(t+\tau) -f(t)\bigr\|^{p}\, dt\Biggr]^{1/p} 
\end{align}
as $l\rightarrow +\infty.$ There exist four possibilities:
\begin{itemize}
\item[1.] $|x|\geq M(\epsilon,\tau)$ and $|x+l|\geq M(\epsilon,\tau).$
\item[2.] $|x|\geq M(\epsilon,\tau)$ and $|x+l|\leq M(\epsilon,\tau).$
\item[3.] $|x|\leq M(\epsilon,\tau)$ and $|x+l|\geq M(\epsilon,\tau).$
\item[4.] $|x|\leq M(\epsilon,\tau)$ and $|x+l|\leq M(\epsilon,\tau).$
\end{itemize}
Let us consider the first case. For $x\geq 0,$ we have
\begin{align*}
\frac{1}{l}& \int_{x}^{x+l}\bigl \| f(t+\tau) -f(t)\bigr\|^{p}\, dt
\\ \leq & \frac{1}{l}\Biggl( \int^{x+1}_{x}\bigl \| f(t+\tau) -f(t)\bigr\|^{p}\, dt+\cdot \cdot \cdot +\int^{x+l}_{x+\lfloor l\rfloor}\bigl \| f(t+\tau) -f(t)\bigr\|^{p}\, dt\Biggr)
\leq \frac{1}{l}l\epsilon^{p}. 
\end{align*}
If $I={\mathbb R}$ and $x\leq 0,$ then $x\leq  -M(\epsilon,\tau),$ $x+l\geq M(\epsilon,\tau)$
and arguing as above we get
\begin{align*}
\frac{1}{l}& \int_{x}^{x+l}\bigl \| f(t+\tau) -f(t)\bigr\|^{p}\, dt
\\ \leq & \frac{1}{l}\int^{-M(\epsilon,\tau)}_{x}\bigl \| f(t+\tau) -f(t)\bigr\|^{p}\, dt
\\ + & \frac{1}{l}
\int_{-M(\epsilon,\tau)}^{M(\epsilon,\tau)}\bigl \| f(t+\tau) -f(t)\bigr\|^{p}\, dt+\frac{1}{l}
\int_{M(\epsilon,\tau)}^{x+l}\bigl \| f(t+\tau) -f(t)\bigr\|^{p}\, dt
\\ \leq & \frac{1}{l}\int^{-M(\epsilon,\tau)}_{x}\bigl \| f(t+\tau) -f(t)\bigr\|^{p}\, dt
\\ + & \frac{2^{p-1}(M(\epsilon,\tau)+2)}{l}\|f\|_{S^{p}}^{p}+
\frac{1}{l}
\int_{M(\epsilon,\tau)}^{l-M(\epsilon,\tau)}\bigl \| f(t+\tau) -f(t)\bigr\|^{p}\, dt
\\ \leq & \frac{\epsilon^{p}}{l}(-M(\epsilon,\tau)-x)+ \frac{2^{p-1}(M(\epsilon,\tau)+2)}{l}\|f\|_{S^{p}}^{p}+\frac{\epsilon^{p}}{l}(l-2M(\epsilon,\tau))
\\
\leq & 2\epsilon^{p}+\frac{2^{p-1}(M(\epsilon,\tau)+2)}{l}\|f\|_{S^{p}}^{p}.
\end{align*}
This implies the existence of a sufficiently large number $l(\epsilon,\tau)>0$ such that the term in \eqref{vasq} is not greater than $\epsilon$ for any $l\geq l(\epsilon,\tau).$ The analysis of cases 2.-4. is analogous and therefore omitted.
\end{proof}

Proposition \ref{weylova-klasa} is motivated by the old results of
A. S. Kovanko \cite[Th\'eor\`eme I, Th\'eor\`eme II]{kovanko-prim}, where the notion of asymptotical almost periodicity has been taken in a slightly different manner.
It is also worth noting that the inclusion $ S^{p}Q-AAP(I : X)\subseteq W_{ap}^{p}(I: X)$ is strict because the space $ W_{ap}^{p}(I: X)$ contains certain Stepanov unbounded functions (see e.g. \cite[Example 4.28]{deda} with $I={\mathbb R}$ and $p=1$). 

Further on, arguing as in the proofs of structural results of \cite[pp. 3-4]{besik}, we may deduce the following:

\begin{thm}\label{krew}
Let $f  : I\rightarrow X$ be q-aap. ($S^{p}$-qaap.). Then we have:
\begin{itemize}
\item[(i)] $cf(\cdot)$ is q-aap. ($S^{p}$-qaap.) for any $c\in {\mathbb C}.$
\item[(ii)] If $X={\mathbb C}$ and $\inf_{x\in {\mathbb R}}|f(x)|=m>0,$ then $1/f(\cdot)$ is q-aap. ($S^{p}$-qaap.).
\item[(iii)] If $(g_{n}: I \rightarrow X)_{n\in {\mathbb N}}$ is a sequence of q-aap. functions and $(g_{n})_{n\in {\mathbb N}}$ converges uniformly to a function $g: I \rightarrow X$, then
$g(\cdot)$ is q-aap..
\item[(iv)] If $(g_{n}: I \rightarrow X)_{n\in {\mathbb N}}$ is a sequence of $S^{p}$-qaap. functions and $(g_{n})_{n\in {\mathbb N}}$ converges to a function $g: I \rightarrow X$ in the space $L_{S}^{p}(I:X)$, then
$g(\cdot)$ is $S^{p}$-qaap..
\item[(v)] The functions $f(\cdot+a)$ and $f(b\, \cdot)$ are likewise q-aap. ($S^{p}$-qaap.),
where $a\in I$ and $b\in I \setminus \{0\}.$
\end{itemize}
\end{thm}

Concerning the pointwise products of (Stepanov) scalar-valued q-aap. functions and (Stepanov) vector-valued q-aap. functions, the following classes play an important role:

\begin{defn}\label{klase}
By $Q_{h}-AAP(I:X)$ ($S^{p}Q_{h}-AAP(I : X)$)
we denote the class consisting of all q-aaa. ($S^{p}$-qaap.) $f : I\rightarrow X$ satisfying that for each $\epsilon>0$ and $\tau \in I$ there exists a finite number $M(\epsilon,\tau)>0$ such that \eqref{dub garden} (\eqref{primer}) holds true.
\end{defn}

The functions from Example \ref{dubgarden121} belong to the class $Q_{h}-AAP(I:X).$ It is also worth noting that the class $S^{p}Q_{h}-AAP(I : X)$ contains all functions that are S-asymptotically $\omega$-periodic in the Stepanov
sense for any number $\omega>0.$
Let $f : I\rightarrow X$ and $g : I \rightarrow {\mathbb C}$ be given. Using the elementary definitions and inequality
$$
\|fg(t+\tau)-fg(t)\|\leq |g(t+\tau)| \|f(t+\tau)-f(t)\| +\|f(t)\| |g(t+\tau)-g(t)|,\ t,\ \tau \in I,
$$
it readily follows the validity of following proposition:

\begin{prop}\label{raw}
Let the functions $f : I\rightarrow X$ and $g : I \rightarrow {\mathbb C}$ be bounded. If $f\in Q_{h}-AAP(I:X)$ and $g\in Q-AAP(I:{\mathbb C})$ ($f\in S^{p}Q_{h}-AAP(I:X)$ and $g\in S^{p}Q-AAP(I:{\mathbb C})$) or $f\in Q-AAP(I:X)$ and $g\in Q_{h}-AAP(I:{\mathbb C})$ ($f\in S^{p}Q-AAP(I:X)$ and $g\in S^{p}Q_{h}-AAP(I:{\mathbb C})$), then we have $fg\in Q-AAP(I:X)$ ($fg\in S^{p}Q-AAP(I:X)$).
\end{prop}

The conclusion established in the above proposition cannot be deduced if the functions $f(\cdot)$ and $g(\cdot)$ belong to the classes of qaap. functions or Stepanov qaap. functions, as the next instructive examples shows:


\begin{example}\label{mangup}
We have that $AP({\mathbb R}:{\mathbb C}) \cdot  SAP_{2}({\mathbb R}: {\mathbb C})$ is not a subset of $Q-AAP({\mathbb R} :{\mathbb C})$ and, in particular, $Q-AAP({\mathbb R} : {\mathbb C} ) \cdot Q-AAP({\mathbb R} : {\mathbb C} )$ is not a subset of $Q-AAP({\mathbb R} :{\mathbb C})$. A typical example of function belonging to the space $[AP({\mathbb R}:{\mathbb C}) \cdot  SAP_{2}({\mathbb R}: {\mathbb C})] \cap S^{p}AAP({\mathbb R} : {\mathbb C})$ but not to the space $Q-AAP({\mathbb R} :{\mathbb C})$ is the function $\cos(\sqrt{2}\pi \cdot)f(\cdot),$ which can be verified as for the function $fg(\cdot)$ considered below, but much simpler.

Assume that $\alpha, \ \beta \in {\mathbb R}$ and $\alpha \beta^{-1}$ is a well-defined irrational number. Then the functions
$$
g_{0}(t)=\sin\Bigl(\frac{1}{2+\cos \alpha t + \cos \beta t}\Bigr),\ t\in {\mathbb R}
$$
and
$$
g(t)=\cos\Bigl(\frac{1}{2+\cos \alpha t + \cos \beta t}\Bigr),\ t\in {\mathbb R}
$$
are Stepanov $p$-ap. but not ap.. These functions are bounded continuous, not uniformly continuous and cannot be q-aap. because they are also aa. (see Theorem \ref{firstreuslt} and \cite{nova-mono} for more details). Consider now the case that $\alpha =\pi$ and $\beta =\sqrt{2}\pi$ for function $g(\cdot).$ 
Let the function $f(\cdot)$ be defined on the real line by zero outside the non-negative real axis and by $f(t):=f_{\overline{\{1/n+1\}}}(t),$ $t\geq 0,$ where the function $f_{\overline{\{1/n+1\}}}(\cdot)$ has the same meaning in \cite[Example 17]{xie}. That is,
we define $f_{\overline{\{1/n+1\}}}(\cdot)$ by $f_{\overline{\{1/n+1\}}}(t):=0$ for $t\in \{0,2, 2n+1-\frac{1}{n+1},2n+1+\frac{1}{n+1} : n\in {\mathbb N}\},$ $f_{\overline{\{1/n+1\}}}(t):=1$ for $t\in 2{\mathbb N}+1$ and linearly outside. 
Assume that for each $\epsilon>0$ (S) holds with $I= [0,\infty)$ and a number $M(\epsilon,\tau)>0$ satisfying \eqref{dub garden}. If $\tau \notin 2{\mathbb N},$ then there exist two integers $k \in {\mathbb N}_{0},$ $n_{0}\in {\mathbb N}$ and a real number $d\in (0,2)$ such that $\tau=2k+d$ and
$d>2/1+n$ for all $n\geq n_{0}.$ Take now any number $s=2n+1,$ where
$n\in {\mathbb N}$ is chosen so that  
$2n+1\geq M(\epsilon,\tau)$
and 
\begin{align}\label{gejaku}
\Biggl| \cos\Bigl( \frac{1}{1+\cos \sqrt{2}(2n+1)\pi}\Bigr) \Biggr|> \frac{4\epsilon}{3}.
\end{align}
The existence of such a number can be easily shown. Then 
$f(2n+1)\geq 3/4$ and $2n+1+\tau \in [2(n+k)+\frac{1}{n+k+1}, 2(n+k+1)+1-\frac{1}{n+k+1}]$ due to the inequality $d>2/1+n.$ In combination with \eqref{gejaku}, the above yields $fg(2n+1+\tau)=0$ and $|fg(2n+1)|> \epsilon,$ which is a contradicition since \eqref{dub garden} holds with $t=2n+1.$ If $\tau=2k$ for some $k\in {\mathbb N},$ then there exists a sufficiently large $s_{0}(\epsilon)>0$ such that
\begin{align*}
|fg(s+2k)-fg(s)|& \geq |f(s+2k)| |g(s+2k)-g(s)| - |f(s+2k)-f(s)|
\\ \geq & |f(s+2k)| |g(s+2k)-g(s)| -\frac{\epsilon}{2},\ s\geq s_{0}(\epsilon).
\end{align*}
If $n\in {\mathbb N}$ is arbitrary and $s=2n+1\geq s_{0}(\epsilon),$ the above estimate yields
\begin{align}\label{fdg}
|fg(2n+1+2k)-fg(2n+1)| \geq \frac{3}{4} |g(2n+1+2k)-g(2n+1)| - \frac{\epsilon}{2}.
\end{align}
Further on, it is very elementary to prove that there exists a strictly increasing sequence $(2a_{l,k}+1)_{l\in {\mathbb N}}$ of odd integers such that the inequality
$$
\Biggl| \cos\Bigl( \frac{1}{1+\cos \sqrt{2}(2a_{l,k}+1+2k)\pi}\Bigr)-\cos\Bigl( \frac{1}{1+\cos \sqrt{2}(2a_{l,k}+1)\pi}\Bigr) \Biggr|\geq \frac{8\epsilon}{3}
$$
holds 
provided $l\in {\mathbb N}.$ In combination with \eqref{fdg}, we get that \eqref{dub garden} does not hold with $s=2n+1.$
Hence, $fg \notin Q-AAP({\mathbb R} : {\mathbb C})$ while in the meantime $g\in S^{p}Q-AAP({\mathbb R} : {\mathbb C}) \setminus Q-AAP({\mathbb R} : {\mathbb C})$ and $f\in SAP_{2}({\mathbb R} : {\mathbb C}).$ 
Since the function $f(\cdot)$ is Stepanov $p$-vanishing, i.e, $\lim_{t\rightarrow +\infty}\int^{t+1}_{t}|f(s)|^{p}\, ds=0,$ it can be easily seen that $fg\in S^{p}AAP([0,\infty): X),$ as well.
\end{example}

\begin{example}\label{prcko-qaz}
It is very simple to illustrate that the pointwise product of two essentially bounded functions from $S^{p}Q-AAP(I : {\mathbb C})$ need not belong to the same class. We will show this only in the case that $I=[0,\infty)$ by giving an example of a Stepanov $p$-ap. function $g(\cdot)$ and a function $f\in SAP_{4}([0,\infty) : {\mathbb C})$ such that $fg\notin S^{p}Q-AAP([0,\infty) : {\mathbb C}).$ 
To see this, put $g(t):=\mbox{sign}(\sin t),$ $t\geq 0,$ where sign$(0):=0.$ Then it is well known that $g(\cdot)$ is Stepanov $p$-ap. function; see e.g. \cite{nova-mono}. We construct $f(\cdot)$ in the following way:
Define $f(t):=0$ for $t\in \{0,43/10,46/10, 4n+\frac{1}{2}-\frac{1}{4n+1},4n+\frac{3}{2}+\frac{1}{4n+1} : n\in {\mathbb N}\},$ $f(t):=1$ for $t\in \cup_{n\in {\mathbb N}}[4n+1/2,4n+3/2]$ and linearly outside. Then it can be easily seen that $f\in SAP_{4}([0,\infty) : {\mathbb C}).$ Furthermore, the function $fg(\cdot)$ cannot be in class $S^{p}Q-AAP([0,\infty) : {\mathbb C})$
because if we suppose the contrary, then we can always take the segment $[4n+1/2,4n+3/2]\subseteq [M(\epsilon,\tau),\infty)$ sufficiently large, with the meaning clear, and the condition \eqref{primer} will be always violated with $t=4n+1/2.$ This can be seen by considering separately two possible cases: $\tau \in 4{\mathbb N}$ and $\tau \notin 4{\mathbb N}.$ In the first case, we have the existence of a number $k\in {\mathbb N}$ such that $\tau=4k.$ Then
\begin{align*}
\int^{4n+\frac{3}{2}}_{4n+\frac{1}{2}}&  \bigl|\mbox{sign}(\sin (s+\tau))f(s+\tau)-\mbox{sign}(\sin s)f(s)\bigr|^{p}\, ds 
\\ = &
\int^{4n+\frac{3}{2}}_{4n+\frac{1}{2}} \bigl|\mbox{sign}(\sin (s+\tau))-\mbox{sign}(\sin s)\bigr|^{p}\, ds 
\end{align*}
for all $n\in {\mathbb N}.$ Further on, observe that 
$$
\sin (s+\tau)-\sin s=2\sin 2k \cos (s+2k),\ s\in {\mathbb R}.
$$
If $\sin 2k>0,$ the terms $\sin(s+\tau)$ and $\sin s$ will have different signs for all $s\in [4n+1/2,4n+3/2]$ provided that there exists a natural number $m\in {\mathbb N}$ such that $4n+2k+1/2$ and $4n+2k+3/2$ belong to the set $(\pi/2+2m\pi,3\pi/2+2m\pi).$ This could happen for arbitrarily large values of $n\in {\mathbb N},$ so that \eqref{primer} does not hold. The examination is similar provided that $\sin 2k<0.$

In the second case, let $\tau=4m+\tau_{0}$ for some $m\in {\mathbb N}_{0}$ and $\tau_{0}\in (0,4).$ 
Since the integer multiples of $\pi$ get arbitrarily close to the integers, there is a strictly increasing sequence of natural numbers $(n_{k})_{k\in {\mathbb N}}$
such that $[4n_{k}+1/2,4n_{k}+3/2]\subseteq \cup_{l\in {\mathbb N}}(2l\pi,(2l+1)\pi)$. For $n=n_{k}\geq 
M(\epsilon,\tau),$ we have $\mbox{sign}(\sin s)=1$ for all $s\in [4n_{k}+1/2,4n_{k}+3/2]$ and
we can use the estimate
\begin{align}
\notag \int^{4n_{k}+\frac{3}{2}}_{4n_{k}+\frac{1}{2}}& \bigl|\mbox{sign}(\sin (s+\tau))f(s+\tau)-\mbox{sign}(\sin s)\bigr|^{p}\, ds 
\\ \notag \geq & 2^{p-1}  \int^{4n_{k}+\frac{3}{2}}_{4n_{k}+\frac{1}{2}}\Bigl[1-\bigl|f(s+\tau)\bigr|^{p}\Bigr]\, ds
\\ \label{smrcu}= & 2^{p-1}\Biggl[1 -\int^{4n_{k}+\tau+\frac{3}{2}}_{4n_{k}+\tau+\frac{1}{2}}\bigl|f(s)\bigr|^{p}\, ds\Biggr] .
\end{align}
After proving that there exist a zero sequence $(a_{k})_{k\in {\mathbb N}}$ and a positive constant $c(\tau_{0})\in (0,1)$ such that $$
\int^{4(n_{k}+m)+\tau_{0}+3/2}_{4(n_{k}+m)+\tau_{0}+1/2}|f(s)|^{p}\, ds >c(\tau_{0})+a_{k}
$$ 
for all $k\in {\mathbb N}$ such that $n_{k}\geq M(\epsilon,\tau),$ the violation od \eqref{primer} becomes apparent. 
\end{example}

It is clear that the sets $Q_{h}-AAP(I : X)$ and $S^{p}Q_{h}-AAP(I : X)$ equipped with the usual operations of pointwise sums and products with scalars form vector spaces. This is no longer true for the spaces $Q-AAP(I : X)$ and $S^{p}Q-AAP(I : X),$ as the following example shows:

\begin{example}\label{primerusa}
For the sake of simplicity, we will consider only the case that $I=[0,\infty).$ Let $f(\cdot)$ and $g(\cdot)$ be as in the former example. Repeating the same argumentation (in the case that $\tau \notin 4{\mathbb N},$ then we can use the estimate $|\mbox{sign}(\sin (s+\tau))+f(s+\tau)-\mbox{sign}(\sin s)-f(s)|=|\mbox{sign}(\sin (s+\tau))+f(s+\tau)-2|\geq 1-|f(s+\tau)|$ for all $s\in [4n_{k}+1/2,4n_{k}+3/2]\subseteq \cup_{l\in {\mathbb N}}(2l\pi,(2l+1)\pi)$), we get that $f+g\notin S^{p}Q-AAP([0,\infty) : {\mathbb C}),$ so that $S^{p}Q-AAP([0,\infty) :  {\mathbb C})+S^{p}Q-AAP([0,\infty) :  {\mathbb C})$ is not contained in $S^{p}Q-AAP([0,\infty) :  {\mathbb C}).$ 
Now we will prove by a simple indirect proof that $AP([0,\infty) :  {\mathbb C})+SAP_{4}([0,\infty) :  {\mathbb C})$ is not contained in class $S^{p}Q-AAP([0,\infty) :  {\mathbb C})$ so that $Q-AAP([0,\infty) : X)$ cannot be a vector space, as announced above. There exists a sequence $(g_{n})_{n\in {\mathbb N}}$ converging to $g(\cdot)$ in $L^{p}_{S}([0,\infty) : {\mathbb C}).$ This implies that $g_{n}+f$ converges to $g+f$ in $L^{p}_{S}([0,\infty) : {\mathbb C})$ as $n\rightarrow \infty.$ Due to Theorem \ref{krew}(iv), there exists a number $n_{0}\in {\mathbb N}$ such that $g_{n_{0}}+f$ is not in class $S^{p}Q-AAP([0,\infty) :  {\mathbb C}).$ The interested reader may try to provide some concrete examples here.
\end{example}

The compactness in the spaces of (equi-)Weyl-$p$-almost periodic functions has been analyzed in \cite{53}-\cite{54} with the help of Lusternik type theorems. It is without scope this paper to analyze similar problems for the space of q-aap. functions and its Stepanov generalizations.

\subsection{Quasi-asymptotically almost periodic functions depending on two parameters and composition principles}\label{ne-to}

The main aim of this subsection is to investigate q-aap. functions depending on two parameters and compositions of q-aap. functions. We start with the folowing definition:

\begin{defn}\label{two-par}
Suppose that $F : I \times X \rightarrow Y$ is a continuous function. Then we say that $F(\cdot,\cdot)$ is quasi-asymptotically almost periodic,
uniformly on bounded subsets of $X,$ iff for each $\epsilon>0$ (S) holds with a number $M(\epsilon,\tau)>0$ such that for each bounded subset $B$ of $X$ we have:
\begin{align*}
\bigl\| F(t+\tau ,x)-F(t,x)\bigr\|_{Y}\leq \epsilon,\mbox{ provided }t\in I, \ x\in B\mbox{ and } |t|\geq M(\epsilon,\tau).
\end{align*}
Denote by $Q-AAP(I \times  X : Y)$ the set consisting of all quasi-asymptotically almost periodic functions from $I \times X$ into $Y.$
\end{defn}

Arguing as in the proofs of \cite[Theorem 3.30, Theorem 3.31]{diagana}, we may deduce the following results about compositions of q-aap. functions:

\begin{thm}\label{prvi-comp}
Suppose that $F\in Q-AAP(I \times  X : Y)$ and $f\in Q-AAP(I : X).$ If there exists a finite number $L>0$ such that
\begin{align}\label{razmisli}
\bigl\|F(t,x)-F(t,y)\bigr\|_{Y}\leq L\|x-y\|,\ x,\ y\in X,\ t\in I,
\end{align}
then the function $t\mapsto F(t,f(t)),$ $t\in I$ belongs to the class $Q-AAP(I : Y).$
\end{thm}

\begin{thm}\label{drugi-comp}
Suppose that $F\in Q-AAP(I \times  X : Y)$ and $f\in Q-AAP(I : X).$ If the function $x\mapsto F(t,x),$ $t\in I$ is uniformly continuous on every bounded subset $B\subseteq X$ uniformly for $t\in I,$
then the function $t\mapsto F(t,f(t)),$ $t\in I$ belongs to the class $Q-AAP(I : Y).$
\end{thm}

The class of $S^{p}$-qaap. functions depending on two parameters is introduced in the following definition:

\begin{defn}\label{two-parr}
Suppose that a function $F : I \times X \rightarrow Y$ satisfies that for each $x\in X$ the function $t\mapsto F(t,x),$ $t\in I$ is Stepanov $p$-bounded. Then we say that $F(\cdot,\cdot)$ is Stepanov $p$-quasi-asymptotically almost periodic,
uniformly on bounded subsets of $X,$ iff for each $\epsilon>0$ (S) holds with a number $M(\epsilon,\tau)>0$ such that for each bounded subset $B$ of $X$ we have:
\begin{align*}
\int^{t+1}_{t}\bigl\|F(s+\tau,x)-F(s,x)\bigr\|^{p}\, ds \leq \epsilon^{p}, & \mbox{ provided }t\in I, \ x\in B\mbox{ and } |t|\geq M(\epsilon,\tau).
\end{align*}
Denote by $S^{p}Q-AAP(I \times  X : Y)$ the set consisting of all Stepanov $p$-quasi-asymptotically almost periodic functions from $I \times X$ into $Y.$
\end{defn}

In \cite[Definition 3.1]{irkutsk-prim}, we have recently introduced the class $e-W_{ap,K}(I\times X,X)$ consisting of equi-Weyl-$p$-ap. functions, uniformly with respect to compact subsets of $X.$ The class $e-W_{ap,K}(I\times X,Y)$ with two different pivot spaces $X$ and $Y$ as well as the class $(e-)W_{ap,B}(I\times X,Y)$ consisting of all (equi-)Weyl-$p$-ap. functions, uniformly with respect to bounded subsets of $X,$ can be introduced in a similar way. Following the method proposed in the proof of Proposition \ref{weylova-klasa}, we can show then $S^{p}Q-AAP(I \times  X : Y)\subseteq W_{ap,B}(I\times X,Y).$

The following composition principles can be deduced in exactly the same way as it has been done in the proofs of \cite[Lemma 2.1, Theorem 2.2]{comp-adv}:

\begin{thm}\label{vcb-show}
Suppose that the following conditions hold:
\begin{itemize}
\item[(i)] $F \in S^{p}Q-AAP(I \times X : Y)  $ with  $p > 1, $ and there exist a number  $ r\geq \max (p, p/p -1)$ and a function $ L_{F}\in L_{S}^{r}(I) $ such that:
\begin{align*}
\|F(t,x)-F(t,y)\| \leq L_{F}(t)\|x-y\|_{Y},\quad t\in I,\ x,\ y\in Y;
\end{align*}
\item[(ii)] $x \in S^{p}Q-AAP (I: X),$ and there exists a set ${\mathrm E} \subseteq I$ with $m ({\mathrm E})= 0$ such that
$ K :=\{x(t) : t \in I \setminus {\mathrm E}\}$
is relatively compact in $X;$ here, $m(\cdot)$ denotes the Lebesgue measure.
\end{itemize}
Then $q:=pr/p+r \in [1, p)$ and $F(\cdot, x(\cdot)) \in S^{q}Q-AAP(I : Y).$
\end{thm}

\begin{thm}\label{vcb-primex}
Suppose that the following conditions hold:
\begin{itemize}
\item[(i)] $F \in S^{p}Q-AAP(I \times X : Y)  $ with  $p \geq 1, $ $L>0$ and 
\begin{align*}
\| F(t,x)-F(t,y)\| \leq L\|x-y\|,\quad t\in I,\ x,\ y\in X.
\end{align*}
\item[(ii)] $x \in S^{p}Q-AAP(I:X),$ and there exists a set ${\mathrm E} \subseteq I$ with $m ({\mathrm E})= 0$ such that
$ K =\{x(t) : t \in I \setminus {\mathrm E}\}$
is relatively compact in $ Y.$
\end{itemize}
Then $F(\cdot, x(\cdot)) \in S^{q}Q-AAP(I : Y).$
\end{thm}

\section{Invariance of quasi-asymptotical almost periodicity under the action of convolution products}\label{nedaju}

Concerning the invariance of quasi-asymptotical almost periodicity under the action of finite convolution product, we have the following result:

\begin{prop}\label{finajt}
Suppose that $(R(t))_{t>0}\subseteq L(X,Y)$ is a strongly continuous operator family and $\int^{\infty}_{0}\|R(s)\|\, ds<\infty.$ If $f\in Q-AAP([0,\infty) : X),$ then the function $F(\cdot)$, defined through \eqref{espebouvi},
is in class $f\in Q-AAP([0,\infty) : Y)$.
\end{prop}

\begin{proof}
Without loss of generality, we may assume that $X=Y.$ It is clear that $\|F(t)\|=\|\int^{t}_{0}R(s)f(t-s)\, ds\|\leq \|f\|_{\infty}\int^{\infty}_{0}\|R(s)\|\, ds,$ $ t\geq 0$ so that $F(\cdot)$ is bounded. Since
$$
\|F(t)-F(t')\|\leq \int^{\infty}_{0}\|R(s)\| \| f(t-s)-f(t'-s)\|\, ds
$$
for any $t,\ t'\geq 0,$ the continuity of $F(t)$ for $t\geq 0$ follows from the boundedness of $f(\cdot)$ and the dominated convergence theorem.
Let $\epsilon>0$ be given. Then (S) holds with a number $M(\epsilon,\tau)>0$ such that
\eqref{dub garden} holds. On the other hand, the condition $\int^{\infty}_{0}\|R(s)\|\, ds<\infty$ implies $\lim_{t\rightarrow +\infty}\int^{\infty}_{t}\|R(s)\|\, ds =0$ so that there exists a finite number $M_{0}(\epsilon)>0$ such that
$\int^{\infty}_{t}\|R(s)\|\, ds<\epsilon$ for any $t\geq M_{0}(\epsilon).$
Let $t\geq M(\epsilon,\tau)+ M_{0}(\epsilon).$ Then we have
\begin{align*}
\|F(t+\tau)&-F(t)\|
\\ = & \Biggl\|\int^{t}_{0}R(s)[f(t+\tau-s)-f(t-s)]\, ds +\int^{t+\tau}_{t}R(s)f(t+\tau-s)\, ds\Biggr \|
\\ \leq & \int^{t}_{0}\|R(s)\| \|f(t+\tau-s)-f(t-s)\|\, ds+\|f\|_{\infty}\int^{t+\tau}_{t}\|R(s)\|\, ds
\\ \leq &  \int^{t-M(\epsilon,\tau)}_{0}\|R(s)\| \|f(t+\tau-s)-f(t-s)\|\, ds
\\ + & \int_{t-M(\epsilon,\tau)}^{t}\|R(s)\| \|f(t+\tau-s)-f(t-s)\|\, ds+\epsilon \|f\|_{\infty}
\\ \leq & \epsilon \int^{\infty}_{0}\|R(s)\|\,ds +2\|f\|_{\infty}\epsilon+\epsilon \|f\|_{\infty},\ t\geq 0,
\end{align*}
which completes the proof in a routine manner.
\end{proof}

The situation is quite similar for the infinite convolution product:

\begin{prop}\label{finajtt}
Suppose that $(R(t))_{t>0}\subseteq L(X,Y)$ is a strongly continuous operator family and $\int^{\infty}_{0}\|R(s)\|\, ds<\infty.$ If $f\in Q-AAP({\mathbb R} : X),$ then the function ${\bf F}(t),$ defined through \eqref{wer},
is in class $f\in Q-AAP({\mathbb R} : Y)$.
\end{prop}

\begin{proof}
Without loss of generality, we may assume that $X=Y.$ The boundedness and continuity of ${\bf F}(\cdot)$ can be proved as in the former proposition. To prove that ${\bf F}(\cdot)$ satisfies the remaining requirement from definition of quasi-asymptotical almost periodicity,
fix a number $\epsilon>0.$ By definition, (S) holds with a number $M(\epsilon,\tau)>0$ satisfying
\eqref{dub garden}. On the other hand, the condition $\int^{\infty}_{0}\|R(s)\|\, ds<\infty$ implies $\lim_{t\rightarrow +\infty}\int^{\infty}_{t}\|R(s)\|\, ds =0$ so that there exists a finite number $M_{0}(\epsilon)>0$ such that
$\int^{\infty}_{t}\|R(s)\|\, ds<\epsilon$ for any $t\geq M_{0}(\epsilon).$
Let $|t|\geq M(\epsilon,\tau)+ M_{0}(\epsilon).$ If $t\leq -M(\epsilon,\tau) -M_{0}(\epsilon),$ then $t\leq -M(\epsilon,\tau),$ $|t-s|\geq  M(\epsilon,\tau)$ for all $s\geq 0,$ and
\begin{align*}
\|{\bf F}(t+\tau)&-{\bf F}(t)\|\leq \int^{\infty}_{0}\|R(s)\| \|f(t+\tau-s)-f(t-s)\|\, ds\leq \epsilon \int^{\infty}_{0}\|R(s)\|\, ds.
\end{align*}
If $t\geq M(\epsilon,\tau) +M_{0}(\epsilon),$ then $t-M(\epsilon,\tau)\geq M_{0}(\epsilon)$ so that $\int^{\infty}_{t-M(\epsilon,\tau)}\|R(s)\|\, ds<\epsilon;$ furthermore, in this case we have:
\begin{align*}
\|{\bf F}(t+\tau)&-{\bf F}(t)\|\leq \int^{\infty}_{0}\|R(s)\| \|f(t+\tau-s)-f(t-s)\|\, ds
\\ \leq & \epsilon \Biggl( \int_{0}^{t-M(\epsilon,\tau)}+\int_{t+M(\epsilon,\tau)}^{\infty}\Biggr)\|R(s)\|\, ds  +2\|f\|_{\infty}\int^{t+M(\epsilon,\tau)}_{t-M(\epsilon,\tau)}\|R(s)\|\, ds
\\ \leq & 2\epsilon \int^{\infty}_{0}\|R(s)\|\, ds  +2\|f\|_{\infty}\int^{\infty}_{t-M(\epsilon,\tau)}\|R(s)\|\, ds
\\ \leq &  2\epsilon \int^{\infty}_{0}\|R(s)\|\, ds  +2\|f\|_{\infty}\epsilon,
\end{align*}
which completes the proof of proposition.
\end{proof}

Suppose that $1/p+1/q=1$ and $\sum^{\infty}_{k=0}\|R(\cdot)\|_{L^{q}[k,k+1]}<\infty.$ This condition implies
$\int^{\infty}_{0}\|R(s)\|\, ds<\infty$ and can be used for the examination of q-aap. properties of convolution products with Stepanov $S^{p}$-qaap. inhomogenities $f(\cdot).$ Keeping in mind the proofs of Proposition \ref{finajt} and Proposition \ref{finajtt}, as well as the proofs of Proposition 2.6.11 and Proposition 3.5.3 in \cite{nova-mono}, the following results can be deduced:

\begin{prop}\label{finajt-rosi}
Suppose that $1/p+1/q=1,$ $(R(t))_{t>0}\subseteq L(X,Y)$ is a strongly continuous operator family and $\sum^{\infty}_{k=0}\|R(\cdot)\|_{L^{q}[k,k+1]}<\infty.$ If $f\in S^{p}Q-AAP([0,\infty) : X),$ then the function $F(\cdot)$, defined by \eqref{espebouvi},
is in class $f\in Q-AAP([0,\infty) : Y)$.
\end{prop}

\begin{prop}\label{finajtt-rosi}
Suppose that $1/p+1/q=1,$ $(R(t))_{t>0}\subseteq L(X,Y)$ is a strongly continuous operator family and $\sum^{\infty}_{k=0}\|R(\cdot)\|_{L^{q}[k,k+1]}<\infty.$ If $f\in S^{p}Q-AAP({\mathbb R} : X),$ then the function ${\bf F}(\cdot)$, defined by \eqref{wer},
is in class $f\in Q-AAP({\mathbb R} : Y)$.
\end{prop}

For the sake of completeness, we will include the proofs (the preassumption $X=Y$ can be made):\vspace{0.2cm}

\noindent {\it Proof of Proposition \ref{finajt-rosi}}. It is easy to see that
$
\|F(t)\|\leq \sum^{\infty}_{k=0}\|R(\cdot)\|_{L^{q}[k,k+1]}\|f\|_{S^{p}},$ $t\geq 0.
$
The
continuity of $F(t)$ for $t\geq 0$ can be proved as follows. Let $t,\ t'\geq 0$ and $|t-t'|\leq 1.$ Then we have
\begin{align*}
& \| F(t)-F(t')\| \leq \int^{t}_{0}\|R(s)\| \|f(t-s)-f(t'-s)\|\, ds+ \int^{t'}_{t}\|R(s)\| \| f(t'-s)\|\, ds
\\ \leq &\sum_{k=0}^{\lfloor t\rfloor}\|R(\cdot)\|_{L^{q}[k,k+1]}\|f(t-\cdot)-f(t'-\cdot)\|_{L^{p}[k,k+1]}+\|f\|_{S^{p}}\|R(\cdot)\|_{L^{q}[\min(t,t'),\max(t,t')]}.
\end{align*}
Since $f\in L^{p}_{loc}([0,\infty) : X),$ we have $\lim_{t'\rightarrow t}\|f(t-\cdot)-f(t'-\cdot)\|_{L^{p}[k,k+1]}=0$ for $k=0,\cdot \cdot \cdot, \lfloor t\rfloor.$ Clearly, $\lim_{t'\rightarrow t}\|R(\cdot)\|_{L^{q}[\min(t,t'),\max(t,t')]}=0$ so that the function $F(\cdot)$ is continuous at point $t.$ Let $\epsilon>0$ be given. Then (S) holds with a number $M(\epsilon,\tau)>0$ satisfying
\eqref{primer}. Furthermore, there exists $k_{0}(\epsilon) \in {\mathbb N}$ such that $\sum_{k=k_{0}(\epsilon)}^{\infty}\|R(\cdot)\|_{L^{q}[k,k+1]}<\epsilon.$ For any $t\geq M(\epsilon,\tau)+k_{0}(\epsilon)+1,$ we have
\begin{align*}
\| F(t+\tau)&-F(t)\| \leq \sum_{k=0}^{\lfloor t\rfloor}\|R(\cdot)\|_{L^{q}[k,k+1]}\|f(t+\tau-\cdot)-f(t-\cdot)\|_{L^{p}[k,k+1]}
\\ \leq & \epsilon\sum_{k=0}^{k_{0}(\epsilon)}\|R(\cdot)\|_{L^{q}[k,k+1]}+2^{p-1}\|f\|_{S^{p}}\sum_{k=k_{0}(\epsilon)}^{\infty}\|R(\cdot)\|_{L^{q}[k,k+1]}
\\ \leq &  \epsilon \sum_{k=0}^{\infty}\|R(\cdot)\|_{L^{q}[k,k+1]} +2\|f\|_{S^{p}}\epsilon,
\end{align*}
finishing the proof.
\vspace{0.2cm}

\noindent {\it Proof of Proposition \ref{finajtt-rosi}}. The boundedness and continuity of function ${\bf F}(\cdot)$ can be shown as in the proof of \cite[Proposition 3.5.3]{nova-mono}. Let $\epsilon>0$ be given. Then (S) holds with a number $M(\epsilon,\tau)>0$ satisfying
\eqref{primer}. As above, there exists $k_{0}(\epsilon) \in {\mathbb N}$ such that $\sum_{k=k_{0}(\epsilon)}^{\infty}\|R(\cdot)\|_{L^{q}[k,k+1]}<\epsilon.$ Let $t\in {\mathbb R}$ be such that $|t|\geq M(\epsilon,\tau)+k_{0}(\epsilon)+1.$ Then we have
\begin{align*}
\| F(t+\tau)&-F(t)\| \leq \sum_{k=0}^{\infty}\|R(\cdot)\|_{L^{q}[k,k+1]}\|f(\cdot+\tau)-f(\cdot)\|_{L^{p}[t-(k+1),t-k]}.
\end{align*}
If $t\leq -M(\epsilon,\tau),$ then $[t-(k+1),t-k]\subseteq (-\infty,-M(\epsilon,\tau)]$ for any $k\in {\mathbb N}_{0}$ and the above estimate immediately implies
$
\| F(t+\tau)-F(t)\| \leq \epsilon \sum_{k=0}^{\infty}\|R(\cdot)\|_{L^{q}[k,k+1]}.
$
If $t\geq M(\epsilon,\tau)+k_{0}(\epsilon)+1,$ then $\lfloor t-M(\epsilon,\tau)\rfloor \geq k_{0}(\epsilon)$ so that
\begin{align*}
\| F(t+\tau)&-F(t)\| \leq \sum_{k=0}^{\lfloor t-M(\epsilon,\tau)\rfloor}\|R(\cdot)\|_{L^{q}[k,k+1]}\|f(\cdot+\tau)-f(\cdot)\|_{L^{p}[t-(k+1),t-k]}
\\ + & \sum_{k=\lfloor t-M(\epsilon,\tau)\rfloor}^{\lceil t+M(\epsilon,\tau)\rceil}\|R(\cdot)\|_{L^{q}[k,k+1]}\|f(\cdot+\tau)-f(\cdot)\|_{L^{p}[t-(k+1),t-k]}
\\ + &\sum_{k=\lceil t+M(\epsilon,\tau)\rceil}^{\infty}\|R(\cdot)\|_{L^{q}[k,k+1]}\|f(\cdot+\tau)-f(\cdot)\|_{L^{p}[t-(k+1),t-k]}
\\ \leq & \epsilon \Biggl( \sum_{k=0}^{\lfloor t-M(\epsilon,\tau)\rfloor} +\sum_{k=\lceil t+M(\epsilon,\tau)\rceil}^{\infty}\Biggr)\|R(\cdot)\|_{L^{q}[k,k+1]}+2\|f\|_{S^{p}}\epsilon
\\ \leq & 2\epsilon \sum_{k=0}^{\infty}\|R(\cdot)\|_{L^{q}[k,k+1]}+2\|f\|_{S^{p}}\epsilon,
\end{align*}
finishing the proof.

\section{Applications to abstract nonautonomous differential equations of first order}\label{convol}

Throughout this section, it will be always assumed that
the operator family $A(\cdot)$
satisfies the condition (H1) and the evolution system
$U(\cdot, \cdot)$ generated by $A(\cdot)$ is hyperbolic, i.e., the condition (H2) holds true.

In our recent research studies \cite{aot-besic}-\cite{sarajevo}, the author has considered the existence and uniqueness of generalized almost periodic properties of mild solutions of \eqref{nije-da-nije}-\eqref{srq-finite} and their semilinear analogues. In the formulations and proofs of all structural results from \cite[Section 2]{aot-besic} and \cite[Section 3]{sarajevo},
the essential boundedness of forcing term $f(\cdot)$ has been required as well as certain additional conditions on the generalized almost periodicity of $f(\cdot).$
In contrast to this, in the formulation of the following result, we require the Stepanov $p$-boundedness of
function $f(\cdot)$ for some exponent $p\in [1,\infty);$
by ${\mathcal F}$ we denote a general function space consisted of continuous functions from $[0,\infty)$ into $X.$

\begin{thm}\label{jos-fajnat}
Let $I=[0,\infty),$ $1/p+1/q=1$ and $f\in S^{p}Q-AAP({\mathbb R} : X).$ If $x\in P(0)X \cap \overline{D(A(0))},$
the function $t\mapsto \int^{t}_{0}U(t,s)Q(s)f(s)\, ds,$ $t\geq 0$ is in class ${\mathcal F}$ and for each $\epsilon>0$ \emph{(S)} holds with a number $M(\epsilon,\tau)>0$ such that
\begin{align}\label{prc-nire}
\sum_{k=0}^{\infty}\bigl\| \Gamma(t+\tau,t+\tau-\cdot)-\Gamma(t,t-\cdot)\bigr\|_{L^{q}[k,k+1]} \leq
\epsilon,\ \mbox{ provided }t\geq M(\epsilon,\tau),
\end{align}
then there exists a unique mild solution $u(\cdot)$ of \eqref{srq-finite} and this solution is
in class $Q-AAP([0,\infty) : X)+{\mathcal F}.$
\end{thm}

\begin{proof}
Since $x\in P(0)X \cap \overline{D(A(0))},$ the mapping $t\mapsto U(t,0)x,$ $t> 0$ is continuous, exponentially decaying and satisfies $\lim_{t\rightarrow 0+}U(t,0)x=x$ (\cite{schnaubelt}). 
The continuity of function $u(t)=U(t,0)x+\int^{t}_{0}U(t,s)f(s)\, ds,$ $ t\geq 0,$ given by \eqref{ndn-finite},
can be deduced as in the proof of \cite[Theorem 2.1]{aot-besic}, since any of the considered terms in the corresponding part of proof of
above-mentioned result can be majorized in a similar way, by using the $S^{p}$-boundedness of function $f(\cdot)$ and the
H\" older inequality. Clearly, $u(t)=\int^{t}_{0}\Gamma(t,s)f(s)\, ds+\int^{t}_{0}U(t,s)Q(s)f(s)\, ds,$ $t\geq 0$
and, by our preassumption made, it suffices to show that the function $t\mapsto \int^{t}_{0}\Gamma(t,s)f(s)\, ds,$ $t\geq 0$
is in class $Q-AAP([0,\infty) : X).$ Applying the H\" older inequality, the estimate \eqref{srq} and the
$S^{p}$-boundedness of function $f(\cdot),$
we get that there exists a finite positive constant $M''>0$ such that
\begin{align*}
& \| u(t) \|\leq M'\|f\|_{S^{p}}\sum_{k=0}^{\infty}\Bigl\| e^{-\omega |\cdot|}\Bigr\|_{L^{q}[t-(k+1),t-k]}
\leq  M'\|f\|_{S^{p}}\sum_{k=0}^{\infty}\Bigl\| e^{-\omega |\cdot|}\Bigr\|_{L^{\infty}[t-(k+1),t-k]}
\\ \leq & M'\|f\|_{S^{p}}\sum_{k=0}^{\infty}\Bigl[e^{-\omega |t-k|}+e^{-\omega |t-k-1|}\Bigr]
\leq  M'\|f\|_{S^{p}}e^{\omega t}\sum_{k=0}^{\infty}\Bigl[e^{-\omega k}+e^{-\omega (k+1)}\Bigr]
\leq M''e^{\omega |t|},
\end{align*}
for any $ t\geq 0.$
Fix a number $\epsilon>0.$ By definition, (S) holds with a number $M(\epsilon,\tau)>0$ satisfying
\eqref{primer}. Keeping in mind the estimate \eqref{prc-nire}, the final conclusion follows from the computation
\begin{align*}
& \|u(t+\tau) -u(t)\|
\\ \leq & \int^{t}_{0}\|\Gamma(t+\tau,t+\tau-s)-\Gamma(t,t-s)\| \|f(t-s)\|\, ds
\\ + & \int^{t+\tau}_{t}\|\Gamma(t+\tau,t+\tau-s)\|
\| f(t+\tau-s)-f(t-s) \|\,ds
\\ \leq & \|f\|_{S^{p}}\sum_{k=0}^{\lfloor t\rfloor}\bigl\| \Gamma(t+\tau,t+\tau-\cdot)-\Gamma(t,t-\cdot)\bigr\|_{L^{q}[k,k+1]}
+2\|f\|_{S^{p}}\sum_{k=\lfloor t\rfloor}^{\lceil t+\tau \rceil}\Bigl\|e^{-\omega |\cdot|}\Bigr\|_{L^{q}[k,k+1]}
\\ \leq & \|f\|_{S^{p}}\sum_{k=0}^{\infty}\bigl\| \Gamma(t+\tau,t+\tau-\cdot)-\Gamma(t,t-\cdot)\bigr\|_{L^{q}[k,k+1]}
+2\|f\|_{S^{p}}\sum_{k=\lfloor t\rfloor}^{\infty}\Bigl\|e^{-\omega |\cdot|}\Bigr\|_{L^{q}[k,k+1]}
\end{align*}
and the obvious equality $\lim_{t\rightarrow \infty}\sum_{k=\lfloor t\rfloor}^{\infty}\|e^{-\omega |\cdot|}\|_{L^{q}[k,k+1]}=0.$
\end{proof}

\begin{rem}\label{essentiali}
It can be simply shown that \eqref{prc} implies
\begin{align}\label{prc-frc-nire}
\int^{+\infty}_{0}\bigl\| \Gamma(t+\tau,t+\tau-s)-\Gamma(t,t-s)\bigr\|\, ds \leq \epsilon,\ t\geq M(\epsilon,\tau).
\end{align}
If we assume that $f\in L^{\infty}([0,\infty) : X)$ in place of $f\in L^{p}_{S}([0,\infty) : X),$
then the validity of \eqref{prc-frc-nire} in place of \eqref{prc-nire} implies that
$u(\cdot)$ is q-aap..
\end{rem}

Concerning the abstract Cauchy problem \eqref{nije-da-nije}, we have the following result:

\begin{thm}\label{jos}
Let $I={\mathbb R},$ $1/p+1/q=1$ and $f\in S^{p}Q-AAP({\mathbb R} : X).$ If for each $\epsilon>0$ \emph{(S)} holds with a number $M(\epsilon,\tau)>0$ such that
\begin{align}\label{prc}
\sum_{k\in {\mathbb Z}}\bigl\| \Gamma(t+\tau,t+\tau-\cdot)-\Gamma(t,t-\cdot)\bigr\|_{L^{q}[k,k+1]} \leq \epsilon,\ \mbox{ provided }t\in {\mathbb R}\mbox{ and } |t|\geq M(\epsilon,\tau),
\end{align}
then there exists a unique mild solution $u(\cdot)$ of \eqref{nije-da-nije} and this solution is q-aap..
\end{thm}

\begin{proof}
As in the proof of Theorem \ref{jos-fajnat}, we can deduce that the function
$u(t)=\int^{+\infty}_{-\infty}\Gamma(t,s)f(s)\, ds,$ $ t\in {\mathbb R},$
defined by \eqref{do-koske}, is bounded. The continuity of $u(\cdot)$ can be shown following the lines of
proof of \cite[Theorem 2.1]{aot-besic}. Assume now that $\epsilon>0$ is a given number. Then (S) holds with a number $M(\epsilon,\tau)>0$ satisfying
\eqref{primer}.
It is clear that, for every $t\in {\mathbb R},$ we have:
\begin{align*}
\|u(t+\tau)& -u(t)\| \leq \int_{-\infty}^{t}e^{-\omega |s|}\|f(t+\tau-s)-f(t-s)\|\, ds
\\ + & \int^{\infty}_{t}e^{-\omega |s|}\|f(t+\tau-s)-f(t-s)\|\, ds
\\ + & \int^{\infty}_{-\infty}\|\Gamma (t+\tau,t+\tau-s)-\Gamma(t,t-s)\| \|f(t-s)\|\, ds
\\ \leq & \int_{-\infty}^{t}e^{-\omega |s|}\|f(t+\tau-s)-f(t-s)\|\, ds
\\ + & \int^{\infty}_{t}e^{-\omega |s|}\|f(t+\tau-s)-f(t-s)\|\, ds
\\ + &\|f\|_{S^{p}}\sum_{k\in {\mathbb Z}}\bigl\| \Gamma(t+\tau,t+\tau-\cdot)-\Gamma(t,t-\cdot)\bigr\|_{L^{q}[k,k+1]}.
\end{align*}
Keeping in mind \eqref{prc}, we get that:
\begin{align*}
\|u(t+\tau)& -u(t)\| \leq \int_{-\infty}^{t}e^{-\omega |s|}\|f(t+\tau-s)-f(t-s)\|\, ds
\\ + & \int^{\infty}_{t}e^{-\omega |s|}\|f(t+\tau-s)-f(t-s)\|\, ds +\|f\|_{S^{p}}\epsilon,
\end{align*}
provided $|t|\geq M(\epsilon,\tau).$ By the proof of Proposition \ref{finajtt-rosi}, it follows the existence of a finite real number $M_{1}(\epsilon,\tau)>0$ such that
$$
\int_{-\infty}^{t}e^{-\omega |s|}\|f(t+\tau-s)-f(t-s)\|\, ds<\epsilon,\mbox{ provided }|t|\geq M_{1}(\epsilon,\tau).
$$
On the other hand, there exists an integer $k_{0}(\epsilon)\in {\mathbb N}$ such that $e^{-\omega k}\leq \epsilon$ for all $k\geq k_{0}(\epsilon).$ Let $|t|\geq 2M(\epsilon,\tau)+1+k_{0}(\epsilon).$ For the second addend, we can use the following calculus involving the H\"older inequality, after dividing the interval of integration $(-\infty,0]$ into two subintervals $(-\infty,-M(\epsilon,\tau)]$ and $[-M(\epsilon,\tau),0]$:
\begin{align*}
&\int^{\infty}_{t} e^{-\omega |s|}\|f(t+\tau-s)-f(t-s)\|\, ds
\\ \leq & \epsilon \sum_{k=0}^{\infty}\Bigl\| e^{-\omega |\cdot|}\Bigr\|_{L^{q}[t+M(\epsilon,\tau)+k,t+M(\epsilon,\tau)+k+1]}+\int^{0}_{-M(\epsilon,\tau)}e^{-\omega |s|}\|f(t+\tau-s)-f(t-s)\|\, ds
\\ \leq & \epsilon e^{-\omega|t+M(\epsilon,\tau)|}\sum_{k=0}^{\infty}\Bigl[e^{-\omega k}+e^{-\omega (k+1)}\Bigr]+2\|f\|_{S^{p}}\sum_{k=0}^{\lfloor M(\epsilon,\tau) \rfloor}\Bigl\| e^{-\omega |\cdot|}\Bigr\|_{L^{q}[t+k,t+k+1]}
\\ \leq & \epsilon e^{-\omega k_{0}(\epsilon)}\sum_{k=0}^{\infty}\Bigl[e^{-\omega k}+e^{-\omega (k+1)}\Bigr]+2\|f\|_{S^{p}}\sum_{k=0}^{\lfloor M(\epsilon,\tau) \rfloor}\Bigl[ e^{-\omega|t-(k+1)|}+e^{-\omega |t-k|} \Bigr]
\\ \leq & \epsilon e^{-\omega k_{0}(\epsilon)}\sum_{k=0}^{\infty}\Bigl[e^{-\omega k}+e^{-\omega (k+1)}\Bigr]+2\|f\|_{S^{p}}(1+M(\epsilon,\tau))e^{-\omega(M(\epsilon,\tau)+k_{0}(\epsilon))}
\\ \leq & \epsilon e^{-\omega k_{0}(\epsilon)}\sum_{k=0}^{\infty}\Bigl[e^{-\omega k}+e^{-\omega (k+1)}\Bigr]+2\|f\|_{S^{p}}(1+\omega)e^{-\omega k_{0}(\epsilon)},
\end{align*}
which completes the proof.
\end{proof}

\begin{rem}\label{essential}
The condition \eqref{prc} implies
\begin{align}\label{prc-frc}
\int^{+\infty}_{-\infty}\bigl\| \Gamma(t+\tau,t+\tau-s)-\Gamma(t,t-s)\bigr\|\, ds \leq \epsilon,\ \mbox{ provided }t\in {\mathbb R}\mbox{ and } |t|\geq M(\epsilon,\tau).
\end{align}
If we assume that $f\in L^{\infty}({\mathbb R} : X)$ in place of $f\in L^{p}_{S}({\mathbb R} : X),$
then the validity of \eqref{prc-frc} in place of \eqref{prc} implies that
$u(\cdot)$ is q-aap..
\end{rem}

\begin{rem}\label{2008}
It is worth noting that Theorem \ref{jos-fajnat} and Remark \ref{essentiali}, as well as
Theorem \ref{jos} and Remark \ref{essential}, continue to hold in the case that the operator family $(A(t))_{t\in {\mathbb R}}$ generates an exponentially stable evolution family $(U(t,s))_{t\geq s}$ in the sense of \cite[Definition 3.1]{diagana-2008}; in this case, the condition (H1) need not be satisfied and the condition (H2) holds with $P(t)=I$ and $Q(t)=0,$ $t\in {\mathbb R};$ $\Gamma(t,s)\equiv U(t,s)$.
\end{rem}

\subsection{Semilinear Cauchy problems}\label{polulinearni}
In this subsection, we consider the existence and uniqueness of q-aap. solutions of the abstract Cauchy problems \eqref{petar-ghj} and \eqref{petar-ghj-finite}. 
We first state the following result about the abstract Cauchy problem \eqref{petar-ghj-finite}:

\begin{thm}\label{jos-fajnat-semilinear}
Let $I=[0,\infty),$ the evolution system $U(\cdot,\cdot)$ be
exponentially stable, 
let $x\in P(0)X \cap \overline{D(A(0))}$ and let $F\in Q-AAP([0,\infty) \times X: X).$ 
Suppose that
for each $\epsilon>0$ \emph{(S)} holds with a number $M(\epsilon,\tau)>0$ satisfying \eqref{prc-frc-nire}.
If there exists a finite number $L\in (0,\omega/M')$ such that
\eqref{razmisli} holds,
then there exists a unique mild solution $u(\cdot)$ of \eqref{petar-ghj-finite} belonging to the 
class $Q-AAP([0,\infty) : X).$
\end{thm}

\begin{proof}
As before, the mapping $t\mapsto U(t,0)x,$ $t> 0$ is continuous, exponentially decaying and satisfies $\lim_{t\rightarrow 0+}U(t,0)x=x.$
Let ${\mathcal P} : Q-AAP([0,\infty) : X) \rightarrow Q-AAP([0,\infty) : X)$ be defined through
$$
{\mathcal Pf}(t):=U(t,0)x+\int^{t}_{0}U(t,s)F(s,f(s))\, ds,\ t\geq 0.
$$
We will first show that the mapping ${\mathcal P}$ is well defined.
Since $C_{0}([0,\infty) : X)+Q-AAP([0,\infty) : X)=Q-AAP([0,\infty) : X)$ and $Q-AAP([0,\infty) : X)$ is a complete metric space by Theorem \ref{krew}(iii), it suffices to show that the mapping $t\mapsto
\int^{t}_{0}U(t,s)F(s,f(s))\, ds,$ $t\geq 0$ is in class $Q-AAP([0,\infty) : X).$ Due to Theorem \ref{prvi-comp}, the function $F(\cdot,f(\cdot))$ is in class $Q-AAP([0,\infty) : X);$ since $Q(t)=0$ for all $t\in {\mathbb R},$ the prescribed assumption on the condition \eqref{prc-frc-nire} yields that Theorem \ref{jos-fajnat} (see also Remark \ref{essentiali}) can be applied,
showing that the mapping $t\mapsto \int^{t}_{0}U(t,s)F(s,f(s))\, ds,$ $t\geq 0$
is in class $Q-AAP([0,\infty) : X)$. Furthermore, the condition $L\in (0,\omega/M')$ implies after a simple calculation involving  \eqref{srq} and \eqref{razmisli} that ${\mathcal P}(\cdot)$ is a contraction, so that the final conclusion simply follows by applying the Banach contraction principle.
\end{proof}

We can similarly prove the following result on the abstract Cauchy problem \eqref{petar-ghj}:

\begin{thm}\label{jos-fajnat-semilinearr}
Let $I={\mathbb R},$ the evolution system $U(\cdot,\cdot)$ be
exponentially stable and $F\in Q-AAP({\mathbb R} \times X: X).$ 
Suppose that
for each $\epsilon>0$ \emph{(S)} holds with a number $M(\epsilon,\tau)>0$ satisfying \eqref{prc-frc}.
If there exists a finite number $L\in (0,\omega/2M')$ such that
\eqref{razmisli} holds,
then there exists a unique mild solution $u(\cdot)$ of \eqref{petar-ghj} belonging to the 
class $Q-AAP({\mathbb R} : X).$
\end{thm}

As mentioned in the introductory part, the semilinear Cauchy problems with $S^{p}$-qaap. forcing term $F(\cdot,\cdot)$ cannot be so easily considered because the range of function $x(\cdot),$ appearing in the formulations of Theorem \ref{vcb-show} and Theorem \ref{vcb-primex}, need not be relatively compact.

We close the paper by providing an illustrative example.

\begin{example}\label{dirichlet-jazz}
Let $X:=L^{2}[0,\pi]$ and $\Delta$ denote the Dirichlet Laplacian in $X,$ acting with the domain $H^{2}[0,\pi] \cap H^{1}_{0}[0,\pi];$ then we know that $\Delta$ generates a strongly continuous semigroup $(T(t))_{t\geq 0}$ on $X$, satisfying the estimate $\|T(t)\| \leq e^{-t},$ $t\geq 0.$ Of concern is the following problem
\begin{align}\label{dortmund-jazz}
u_{t}(t,x)&=u_{xx}(t,x)+q(t,x)u(t,x)+f(t,x),\ t\geq 0,\ x\in [0,\pi];
\\\label{dortmund-jazzz} & u(0)=u(\pi)=0,\ u(0,x)=u_{0}(x)\in X,
\end{align}
where $q : {\mathbb R} \times [0,\pi] \rightarrow {\mathbb R}$ is a jointly continuous function satisfying that $q(t,x)\leq -\gamma_{0},$ $(t,x)\in {\mathbb R} \times [0,\pi],$ for some number $\gamma_{0}>0.$ Define
$$
A(t)\varphi:=\Delta \varphi +q(t,\cdot)\varphi,\ \varphi \in D(A(t)):=D(\Delta)=H^{2}[0,\pi] \cap H^{1}_{0}[0,\pi],\ t\in {\mathbb R}.
$$
Then $(A(t))_{t\in {\mathbb R}}$ generates an exponentially stable evolution family $(U(t,s))_{t\geq s}$ in the sense of \cite[Definition 3.1]{diagana-2008}, which is given by
$$
U(t,s)\varphi:=T(t-s)e^{\int^{t}_{s}q(r,\cdot)\, dr}\varphi,\quad t\geq s.
$$
It is clear that we can rewrite the initial value problem \eqref{dortmund-jazz}-\eqref{dortmund-jazzz} in the following form:
$$
u^{\prime}(t)=A(t)u(t)+f(t),\ t\geq 0; \ \ u(0)=u_{0}.
$$
Hence, Theorem \ref{jos-fajnat}, resp. Theorem \ref{jos} (Theorem \ref{jos-fajnat-semilinear}, resp. Theorem \ref{jos-fajnat-semilinearr}), can be applied provided that for each $\epsilon>0$
(S) holds with a number $M(\epsilon,\tau)>0$ such that
the following condition holds:
\begin{align}\label{zagrebin-jazz}
\sum_{k=0}^{\infty}\Biggl\|e^{-|\cdot|}\sup_{x\in [0,\pi]}\Bigl| e^{\int^{t+\tau}_{t+\tau-\cdot}q(r,x)\, dr}-
e^{\int^{t}_{t-\cdot}q(r,x)\, dr} \Bigr| \Biggr\|_{L^{q}[k,k+1]}<\epsilon,\ t\geq M(\epsilon, \tau),
\end{align}
resp.
\begin{align}\label{zagrebin-jazzz}
\sum_{k\in {\mathbb Z}}\Biggl\| e^{-|\cdot|}\sup_{x\in [0,\pi]}\Bigl| e^{\int^{t+\tau}_{t+\tau-\cdot}q(r,x)\, dr}-
e^{\int^{t}_{t-\cdot}q(r,x)\, dr}  \Bigr| \Biggr\|_{L^{q}[k,k+1]}<\epsilon,\ |t|\geq M(\epsilon, \tau).
\end{align}
The conditions \eqref{zagrebin-jazz} and \eqref{zagrebin-jazzz} hold for a wide class of functions $q(\cdot,\cdot)$ and we will prove here that this condition particularly holds for the function $q(t,x):=-\gamma_{0}-3t^{2}-f(x),$ $t\geq 0,$ $x\in [0,\pi],$ where $f : [0,\infty) \rightarrow  [0,\infty)$ is a continuous function (see also \cite[Example 3.1]{aot-besic}, where we have analyzed the same choice). In our concrete situation, we have
\begin{align*}
\sup_{x\in [0,\pi]}& \Bigl| e^{\int^{t+\tau}_{t+\tau-s}q(r,x)\, dr}-
e^{\int^{t}_{t-s}q(r,x)\, dr}  \Bigr|
\\ \leq & \mbox{Const.} \cdot e^{-s^{3}}\Bigl| e^{3s(t+\tau)(s-t-\tau)}-e^{3st(s-t)} \Bigr|,\ t,\ s,\ \tau\geq 0.  
\end{align*}
Using this estimate and the Lagrange mean value theorem, it readily follows that:
\begin{align*}
\Biggl\|& e^{-|\cdot|}\sup_{x\in [0,\pi]}\Bigl| e^{\int^{t+\tau}_{t+\tau-\cdot}q(r,x)\, dr}-
e^{\int^{t}_{t-\cdot}q(r,x)\, dr} \Bigr| \Biggr\|_{L^{\infty}[k,k+1]}
\\ \leq & \mbox{Const.} \cdot |\tau| \sup_{s\in [k,k+1]}\Bigl[ e^{3st(s-t)}+e^{3s(t+\tau)(s-t-\tau)} \Bigr] \cdot \Bigl[ 3(k+1)^{2}+6(k+1)(t+\tau)\Bigr]
\\ \leq & \mbox{Const.} \cdot |\tau| \Bigl[ e^{3kt(k-t)}+e^{3(k+1)t(k+1-t)}+e^{3k(t+\tau)(k-t-\tau)}+ e^{3(k+1)(t+\tau)(k+1-t-\tau)} \Bigr]
\\ \cdot & \Bigl[ 3(k+1)^{2}+6(k+1)(t+\tau)\Bigr],\ t,\ \tau\geq 0,\ k\in {\mathbb N}_{0}.
\end{align*}
Let $3/4<c<1.$ Then $3ts(s-ct)\leq 3s^{3}/4c$ for all $t,\ s\geq 0$ and therefore
\begin{align*}
\\ \leq & \mbox{Const.} \cdot |\tau| \Bigl[ e^{3kt(k-ct)}e^{-3ct^{2}}+e^{3(k+1)t(k+1-ct)}e^{-3ct^{2}}
\\ + & e^{3k(t+\tau)(k-c(t+\tau))}e^{-3c(t+\tau)^{2}}+ e^{3(k+1)(t+\tau)(k+1-c(t+\tau))}e^{-3c(t+\tau)^{2}} \Bigr]
\\ \cdot & \Bigl[ 3(k+1)^{2}+6(k+1)(t+\tau)\Bigr]
\\ \leq & \mbox{Const.} \cdot |\tau| \Bigl[ e^{3k^{3}/4c}e^{-3ct^{2}}+e^{3(k+1)^{3}/4c}e^{-3ct^{2}}
\\ + & e^{3k^{3}/4c}e^{-3c(t+\tau)^{2}}+ e^{3(k+1)^{3}/4c}e^{-3c(t+\tau)^{2}} \Bigr]
\\ \cdot & \Bigl[ 3(k+1)^{2}+6(k+1)(t+\tau)\Bigr]
\\ \leq & \mbox{Const.} \cdot |\tau| e^{3(k+1)^{3}/4c}e^{-3ct^{2}} \Bigl[ 3(k+1)^{2}+6(k+1)(t+\tau)\Bigr],\ t,\ \tau\geq 0.
\end{align*}
Since $3/4c<1,$ the series in \eqref{zagrebin-jazz} is convergent with $q=\infty$ and has a sum which does not exceed
$
\mbox{Const.} \cdot |\tau| e^{-3ct^{2}} (1+t+\tau),$ $ t,\ \tau\geq 0.
$
At the end, it suffices to observe that  for each $\epsilon>0$ and $\tau \geq 0$ 
there exists a finite number $M(\epsilon,\tau)>0$ such that $|\tau| e^{-3ct^{2}} (1+t+\tau)<\epsilon$ for any $t\geq M(\epsilon,\tau).$ This shows that Theorem \ref{jos-fajnat} can be applied with any exponent $p\in [1,\infty).$
\end{example}

\bibliographystyle{amsplain}

\end{document}